\documentclass[a4paper]{amsart}
\usepackage{a4wide}
\usepackage{amsmath,amssymb,amsthm,mathtools}
\usepackage{xcolor}
\usepackage{needspace}
\definecolor{refblue}{RGB}{0,70,160}
\usepackage[colorlinks=true,linkcolor=refblue,citecolor=refblue,urlcolor=refblue]{hyperref}
\hypersetup{pdftitle={Completely symmetric subspaces of noncommutative Lp-spaces},
pdfauthor={Jinghao Huang},
pdfsubject={The Junge--Xu selection problem for exponents above two}}
\newtheorem{theorem}{Theorem}[section]
\newtheorem{lemma}[theorem]{Lemma}
\newtheorem{corollary}[theorem]{Corollary}

\theoremstyle{remark}

\title{Completely symmetric subspaces of noncommutative $L_p$-spaces}
\author[J. Huang]{Jinghao Huang}
\thanks{
This work was supported by the NNSF of China
(Nos. 12301160, 12471134 and 12671159).
}
\address{Institute for Advanced Study in Mathematics of HIT,
Harbin 150001, China}
\email{jinghao.huang@hit.edu.cn}
\begin{document}
\begin{abstract}
Let $2<p<\infty$. We prove that every infinite-dimensional closed
subspace $F$ of a noncommutative $L_p$-space contains a subspace
completely isomorphic to $\ell_p$, $C_p$, $R_p$, or $C_p\cap R_p$
and completely complemented in $F$. This answers the question of
Junge and Xu on the existence
of completely symmetric basic sequences.
\end{abstract}
\maketitle

\section{Introduction}

We begin with the classical subspace theory of $L_p(0,1)$.
Let $2<p<\infty$. Kadec and Pe\l czy\'nski~\cite[Corollary~5]{KP}
proved that
every normalized weakly null sequence in $L_p(0,1)$ has a subsequence
equivalent to the unit vector basis of either $\ell_p$ or $\ell_2$.
Their  result also gives the following alternative for every
infinite-dimensional closed subspace $X\subset L_p(0,1)$
\cite[Corollary~2]{KP}:
\[
 X\cong\ell_2,
 \qquad\text{or}\qquad
 X\text{ contains a copy of }\ell_p\text{ complemented in }L_p(0,1).
\]

Noncommutative versions of this dichotomy
were obtained by Sukochev~\cite[Theorem~2.4]{Sukochev1996} for finite
von Neumann algebras, by Randrianantoanina~\cite[Theorem~3.1]{RandEmb}
in the semifinite setting, and by Raynaud and Xu~\cite[Theorem~0.2]{RX}
for arbitrary von Neumann algebras. For a general von Neumann algebra
$\mathcal M$ and $2<p<\infty$, every closed subspace $X\subset L_p(\mathcal M)$ is
either isomorphic to a Hilbert space and complemented in $L_p(\mathcal M)$,
or contains a subspace isomorphic to $\ell_p$ and complemented in
$L_p(\mathcal M)$.

Rosenthal~\cite{Rosenthal1970} studied the
subspaces generated by
independent mean-zero random variables. For $2<p<\infty$, he
classified the possible isomorphism types of their infinite-dimensional
closed spans as
\[
 \ell_p,\qquad \ell_2,\qquad \ell_p\oplus\ell_2,\qquad X_p,
\]
where $X_p$ is the space introduced in that paper. 

Johnson, Maurey, Schechtman and Tzafriri developed the
finite-dimensional theory of symmetric sequences in $L_p(0,1)$
\cite[Theorem~1.14]{JMST}; see also the quantitative formulation
\cite[Corollary~6.7(i)]{JX}.
In the range
$2<p<\infty$, their estimates have the following form. If
$(x_1,\ldots,x_N)$ is a finite symmetric basic sequence in
$L_p(0,1)$ with symmetry constant at most $C$, then there are numbers
$\alpha,\beta\ge0$ and a constant $K=K(p,C)$, independent of $N$,
such that
\[
 \begin{aligned}
 K^{-1}\max\left\{\alpha\Big(\sum_{k=1}^N|c_k|^p\Big)^{1/p},
                    \beta\Big(\sum_{k=1}^N|c_k|^2\Big)^{1/2}\right\}
 &\le \left\|\sum_{k=1}^N c_kx_k\right\|_p\\
 &\le K\max\left\{\alpha\Big(\sum_{k=1}^N|c_k|^p\Big)^{1/p},
                    \beta\Big(\sum_{k=1}^N|c_k|^2\Big)^{1/2}\right\}
 \end{aligned}
\]
for every $(c_1,\ldots,c_N)\in\mathbb C^N$.
In the noncommutative setting,
the corresponding operator-space estimates involve both column and row
square functions; see~\cite[Theorem~6.2]{JX}.

We now fix the operator-space notation used in the paper. Let
$S_p\coloneqq S_p(\ell_2)$ be the Schatten $p$-class, and let $(e_{ij})$ be
the usual matrix units on $\ell_2$. Define
\[
 C_p\coloneqq\overline{\operatorname{span}}\{e_{k1}:k\in\mathbb N\}\subset S_p,
 \qquad
 R_p\coloneqq\overline{\operatorname{span}}\{e_{1k}:k\in\mathbb N\}\subset S_p.
\]
For a finite family $a=(a_1,\ldots,a_N)\in \mathbb{M}_m^N$, the matrix norms
of their canonical bases are
\begin{equation}\label{eq:column-row-amplification}
\begin{aligned}
 \left\|\sum_{k=1}^N a_k\otimes e_{k1}\right\|_{S_p^m[C_p]}
 &=\left\|\left(\sum_{k=1}^N a_k^*a_k\right)^{1/2}\right\|_{S_p^m},\\
 \left\|\sum_{k=1}^N a_k\otimes e_{1k}\right\|_{S_p^m[R_p]}
 &=\left\|\left(\sum_{k=1}^N a_ka_k^*\right)^{1/2}\right\|_{S_p^m}.
\end{aligned}
\end{equation}
Define
\begin{equation}\label{eq:intersection-definition}
 C_p\cap R_p\coloneqq
 \overline{\operatorname{span}}\{(e_{k1},e_{1k}):k\in\mathbb N\}
 \subset C_p\oplus_\infty R_p.
\end{equation}
Writing $e_k\coloneqq(e_{k1},e_{1k})$ in this space, its matrix norms are
\begin{equation}\label{eq:intersection-matrix-norm}
 \left\|\sum_{k=1}^N a_k\otimes e_k\right\|_{\mathbb M_m[C_p\cap R_p]}
 =\max\left\{
 \left\|\sum_{k=1}^N a_k\otimes e_{k1}\right\|_{\mathbb M_m[C_p]},
 \left\|\sum_{k=1}^N a_k\otimes e_{1k}\right\|_{\mathbb M_m[R_p]}
 \right\}.
\end{equation}
We also write
\begin{equation}\label{eq:rho-definition}
 \rho_p(a)\coloneqq\max\left\{
 \left\|\left(\sum_{k=1}^N a_k^*a_k\right)^{1/2}\right\|_{S_p^m},
 \left\|\left(\sum_{k=1}^N a_ka_k^*\right)^{1/2}\right\|_{S_p^m}
 \right\}.
\end{equation}
For $\ell_p$, the corresponding matrix norm is
\[
 \left\|\sum_{k=1}^N a_k\otimes e_k\right\|_{S_p^m[\ell_p]}
 =\left(\sum_{k=1}^N\left\|a_k\right\|_{S_p^m}^p\right)^{1/p}.
\]
These are the standard matrix norm descriptions of the spaces above;
see~\cite[Section~6, immediately before Corollary~6.3]{JX}. 

Let $\mathcal M$ be a von Neumann algebra, let $2<p<\infty$, and let
$(x_n)\subset L_p(\mathcal M)$ be a basic sequence. Following the
notion introduced by Junge and Xu
\cite[Section~6, immediately preceding Corollary~6.4]{JX}, we call
$(x_n)$ completely symmetric if there is a constant $C\ge1$ such that
\[
 \left\|\sum_{k=1}^N\varepsilon_k a_{\pi(k)}\otimes x_k\right\|_p
 \le C\left\|\sum_{k=1}^N a_k\otimes x_k\right\|_p
\]
for every $m,N\in\mathbb N$, every
$(a_1,\ldots,a_N)\in \mathbb{M}_m^N$, every
$(\varepsilon_1,\ldots,\varepsilon_N)\in\{-1,1\}^N$, and every
permutation $\pi$ of $\{1,\ldots,N\}$. Both norms are taken in
$L_p(\mathbb{M}_m\overline\otimes\mathcal M)$, and $C$ is independent of
$m,N,(a_k),(\varepsilon_k)$, and $\pi$.

We now recall the Schatten amplification norms and the completely bounded
norm. For an operator space $E$, we use the standard vector-valued
Schatten space
$S_p^m[E]\coloneqq(\mathbb M_m[E],S_1^m\widehat\otimes E)_{1/p}$,
where $\widehat\otimes$ denotes the operator-space projective tensor
product and $(\cdot,\cdot)_{1/p}$ denotes complex interpolation.
If $E\subset L_p(\mathcal M)$ is closed and carries the induced
operator-space structure, then for $z\in\mathbb M_m\odot E$ we have
\cite[Section~6, before Lemma~6.1]{RX}
\[
 \left\|z\right\|_{S_p^m[E]}=\left\|z\right\|_{L_p(\mathbb{M}_m\overline\otimes\mathcal M)},
 \qquad
 S_p^m[E]=\overline{\mathbb{M}_m\odot E}^{\,L_p(\mathbb{M}_m\overline\otimes\mathcal M)}.
\]
Here $\mathbb{M}_m\odot E$ denotes the algebraic tensor product.
For arbitrary operator spaces $E,F$ and a linear map $T:E\to F$, the completely
bounded norm is given by Pisier's characterization
\cite[Lemma~1.7, formula~$(1.7)'$, p.~23]{PisierVector}:
\begin{equation}\label{eq:cb-amplification-criterion}
 \left\|T:E\to F\right\|_{\mathrm{cb}}
 =\sup_{m\in\mathbb N}
 \left\|\operatorname{id}_{\mathbb{M}_m}\otimes T:
       S_p^m[E]\longrightarrow S_p^m[F]\right\|.
\end{equation}
A projection $P:F\to F$ is said to be completely bounded when
$\left\|P\right\|_{\mathrm{cb}}<\infty$.

Junge and Xu~\cite[Corollary~6.8]{JX} proved the classification
of completely symmetric subspaces: if
$X\subset L_p(\mathcal M)$ is infinite-dimensional and has a
completely symmetric basis, then
\[
 X\ \text{is completely isomorphic to one of }
\ell_p,\quad C_p,\quad R_p,\quad C_p\cap R_p.
\]
For $2<p<\infty$, these four spaces are pairwise not completely isomorphic.
Indeed, $\ell_p$ is not Hilbertian, whereas the other three are Hilbertian
and are distinguished by their column and row fundamental functions
\cite[Proposition~6.1]{JXHilbertian}; the fact that $C_p$ and $R_p$ are not
completely isomorphic is also explicitly recalled in \cite[p.~20]{JNO}.
\par
This classification is valid for arbitrary von Neumann algebras. However, it does not by itself produce a
completely symmetric sequence in every infinite-dimensional subspace.

Under the additional assumption that $\mathcal M$ is hyperfinite,
Junge and Xu~\cite[Corollary~6.3]{JX} proved that every
semi-normalized weakly null sequence in $L_p(\mathcal M)$ has a
subsequence completely equivalent to the canonical basis of one of the
four spaces listed above. As a consequence, every infinite-dimensional
subspace contains a completely symmetric basic sequence~\cite[Corollary~6.4]{JX}. Immediately
after \cite[Corollary~6.4]{JX} they stated:
\begin{quote}
It is not known whether the hyperfiniteness assumption can be removed
for $2<p<\infty$.
\end{quote}
The present paper answers this question affirmatively. We prove the
following stronger complementation statement.

\begin{theorem}
\label{thm:general-high-p-relative-complement}
Let $\mathcal M$ be an arbitrary von Neumann algebra and
$2<p<\infty$. Every infinite-dimensional closed subspace
$F\subset L_p(\mathcal M)$ contains an infinite-dimensional
subspace $Z$ which is completely complemented in $F$ and completely
isomorphic to one of
\[
 \ell_p,\qquad C_p,\qquad R_p,\qquad C_p\cap R_p.
\]
\end{theorem}

More explicitly, there is a completely bounded projection
\[
 P:F\longrightarrow F,\qquad P^2=P,\qquad P(F)=Z.
\]
Moreover, when $\mathcal M$ is finite, neither $C_p$ nor $R_p$
embeds completely isomorphically into $L_p(\mathcal M)$
\cite[Corollary~6.10]{JX}.
In this case, the projection can be chosen on the whole
space $L_p(\mathcal M)$, and there are only two possible types for the
infinite-dimensional subspace $Z$ selected to be completely complemented in $F$,
namely $\ell_p$ and $C_p\cap R_p$:
\begin{corollary}\label{cor:finite-models}
Let $\mathcal M$ be a finite von Neumann algebra and $2<p<\infty$.
Every infinite-dimensional closed subspace $F\subset L_p(\mathcal M)$
contains an infinite-dimensional subspace $Z$ which is completely
complemented in $L_p(\mathcal M)$ and completely isomorphic to
$\ell_p$ or $C_p\cap R_p$.
\end{corollary}

 The key new ingredient in the proof of
Theorem~\ref{thm:general-high-p-relative-complement} is the matrix-valued upper selection estimate in
Theorem~\ref{thm:general-all-p-upper} below. 
This extends the  selection estimate in the  hyperfinite setting given in
Junge--Xu~\cite[Theorem~6.2 and Corollary~6.3]{JX}
to the setting of 
  an arbitrary von
Neumann algebra.
 
\begin{theorem}[Upper selection estimate]
\label{thm:general-all-p-upper}
For each $p\in(2,\infty)$ there is $c_p<\infty$ with the following
property: for any von Neumann algebra $\mathcal M$ and any bounded weakly
null sequence $(x_n)$ in $L_p(\mathcal M)$, there are integers
$1\le n_1<n_2<\cdots$ such that
\begin{equation}\label{eq:upper-selection-estimate}
 \left\|\sum_{k=1}^N a_k\otimes x_{n_k}\right\|_p
 \le c_p\left(\sup_{n\in\mathbb N}\left\|x_n\right\|_p\right)\rho_p(a)
\end{equation}
for all $m,N\in\mathbb N$ and $a=(a_1,\ldots,a_N)\in \mathbb{M}_m^N$.
The quantity $\rho_p(a)$ is defined in
\eqref{eq:rho-definition}.
The subsequence $(x_{n_k})$ depends only on $(x_n)$ and $p$,
whereas $c_p$ depends only on $p$.
\end{theorem}

\section{Preliminaries}\label{sec:preliminaries}

All spaces considered in this paper are complex, and we denote by
$\mathbb N$ the set $\{1,2,\ldots\}$.

We write $\mathbb{M}_m\coloneqq\mathbb{M}_m(\mathbb C)$,
the space of all $m\times m$ complex matrices, and
$S_p^m\coloneqq(\mathbb{M}_m,\left\|\cdot\right\|_{S_p})$, where
$\left\|a\right\|_{S_p}\coloneqq(\operatorname{tr}|a|^p)^{1/p}$ and
$\operatorname{tr}(a)\coloneqq\sum_{j=1}^m a_{jj}$ for $a\in \mathbb{M}_m$.

We recall Haagerup's construction of noncommutative $L_p$-spaces;
see \cite[Section~1.2]{HJX} and \cite[Sections~9.1--9.2]{HiaiModular}.
Let $\varphi$ be a faithful normal semifinite weight on $\mathcal M$.
Form the continuous core
\[
 \mathcal R\coloneqq\mathcal M\rtimes_{\sigma^\varphi}\mathbb R,
\]
where $(\sigma_t^\varphi)_{t\in\mathbb R}$ is the modular
automorphism group of $\varphi$. Its dual action
$(\widehat\sigma_t)_{t\in\mathbb R}$ and canonical faithful normal
semifinite trace $\tau_{\mathcal R}$ satisfy
\[
 \tau_{\mathcal R}\circ\widehat\sigma_t
 =e^{-t}\tau_{\mathcal R}\qquad(t\in\mathbb R).
\]
Let $L_0(\mathcal R,\tau_{\mathcal R})$ denote the algebra of
closed densely defined $\tau_{\mathcal R}$-measurable operators
affiliated with $\mathcal R$. For $1\le p<\infty$, put
\[
 L_p(\mathcal M)\coloneqq
 \big\{x\in L_0(\mathcal R,\tau_{\mathcal R}):
       \widehat\sigma_t(x)=e^{-t/p}x\quad(t\in\mathbb R)\big\}.
\]
For $\omega\in\mathcal M_*^+$, let $h_\omega$ be the
Radon--Nikod\'ym derivative of its dual weight with respect to
$\tau_{\mathcal R}$. The correspondence $\omega\mapsto h_\omega$
extends to a linear identification $\mathcal M_*\cong L_1(\mathcal M)$.
The canonical functional and the $L_p$-norm are given by
\[
 \operatorname{Tr}(h_\omega)\coloneqq\omega(\mathbf{1})\quad(\omega\in\mathcal M_*),
 \qquad
 \left\|x\right\|_p\coloneqq\big(\operatorname{Tr}(|x|^p)\big)^{1/p}
 \quad(x\in L_p(\mathcal M)).
\]
We set $L_\infty(\mathcal M)\coloneqq\mathcal M$ with the operator norm.
The spaces are independent of $\varphi$ up to canonical isometric
identifications; see \cite[Remark~1.2]{HJX}. 

Products of measurable operators mean their closures. H\"older's
inequality and trace cyclicity take the following forms
\cite[Section~1.2]{HJX}:
\[
 \begin{aligned}
 \left\|xy\right\|_r&\le\left\|x\right\|_p\left\|y\right\|_q,
 &&x\in L_p(\mathcal M),\ y\in L_q(\mathcal M),\quad
   \frac1r=\frac1p+\frac1q,\\
 \operatorname{Tr}(xy)&=\operatorname{Tr}(yx),
 &&x\in L_p(\mathcal M),\ y\in L_{p'}(\mathcal M),\quad
   \frac1p+\frac1{p'}=1,
 \end{aligned}
\]
where all exponents belong to $[1,\infty]$.
If $\mathcal M$ carries a faithful normal semifinite trace $\tau$,
then, for $1\le p<\infty$, $L_p(\mathcal M)$ is canonically
isometric to the tracial space \cite[Remark~1.2]{HJX}
\[
 L_p(\mathcal M,\tau)
 \coloneqq\{x\in L_0(\mathcal M,\tau):\tau(|x|^p)<\infty\},
 \qquad \left\|x\right\|_p=\tau(|x|^p)^{1/p}.
\]
Whenever a faithful normal finite trace $\tau$ on $\mathcal M$
is specified, we use this identification with $L_p(\mathcal M,\tau)$.
For $a_k\in \mathbb{M}_m$ and $x_k\in L_p(\mathcal M)$,
the notation $\left\|\sum_k a_k\otimes x_k\right\|_p$ refers to the norm in
$L_p(\mathbb{M}_m\overline\otimes\mathcal M)$.

Kosaki's interpolation theorem \cite[Theorem~9.1]{Kosaki}, together
with reiteration, gives the weighted compatible couples used in
Subsection~\ref{sec:analytic}.

For $2\le p<\infty$, $m,N\in\mathbb N$ and $a\in \mathbb{M}_m^N$, we
use the notation $\rho_p(a)$ defined in the Introduction.
The column and row square-function terms in $\rho_p(a)$ are,
respectively, the $S_p$ amplification norms of the canonical
bases of $C_p$ and $R_p$. For the standard
operator space intersection, coordinate contractions and the diagonal
embedding into $C_p\oplus_\infty R_p$ give
\begin{equation}\label{eq:intersection-amplification}
 \rho_p(a)\le
 \left\|\sum_{k=1}^N a_k\otimes e_k\right\|_{S_p^m[C_p\cap R_p]}
 \le2\rho_p(a).
\end{equation}
The amplification norm for the canonical basis of $\ell_p$ is
$(\sum_{k=1}^N\left\|a_k\right\|_{S_p^m}^p)^{1/p}$. In particular, all four canonical
bases are completely symmetric.

Let $(r_k)$ be independent Rademacher signs on a probability
space $(\Omega,\mathbb P)$. For each $2\le p<\infty$, fix a
constant $\kappa_p\ge1$, depending only on $p$, such that
\begin{equation}\label{eq:khintchine-general}
 \left(\mathbb E\left\|\sum_{k=1}^N r_kx_k\right\|_p^p\right)^{1/p}
 \le\kappa_p\max\left\{
 \left\|\left(\sum_{k=1}^N x_k^*x_k\right)^{1/2}\right\|_p,
 \left\|\left(\sum_{k=1}^N x_kx_k^*\right)^{1/2}\right\|_p
 \right\}
\end{equation}
for every von Neumann algebra $\mathcal N$, every $N\in\mathbb N$,
and every $x_1,\ldots,x_N\in L_p(\mathcal N)$.
Such a constant exists by the noncommutative Khintchine inequalities
\cite[Theorem~6.2 and Remark~3.2]{HJX}.
In particular, for $m,N\in\mathbb N$ and
$a=(a_1,\ldots,a_N)\in\mathbb{M}_m^N$, the two-sided estimate is
\begin{equation}\label{eq:khintchine}
 \rho_p(a)\le
 \left\|\sum_{k=1}^N r_k\otimes a_k\right\|_{L_p(\Omega;S_p^m)}
 \le\kappa_p\rho_p(a),\qquad 2\le p<\infty.
\end{equation}
We use the upper estimate in Subsection~\ref{sec:analytic}.

We shall use two results of Raynaud and Xu.
By \cite[Corollary~6.3]{RX}, a subspace of $L_p(\mathcal M)$,
$2<p<\infty$, either is Banach-space isomorphic to a Hilbert space
or contains a copy of $\ell_p$ which is completely complemented in
the ambient space. If a closed subspace $F$ contains no $\ell_p$,
the equivalence of (iii) and (v) in \cite[Theorem~5.1]{RX} gives $h\in L_r(\mathcal M)$,
$1/r=1/2-1/p$, such that
\[
 c\left\|x\right\|_p\le
 \big(\left\|xh\right\|_2^2+\left\|hx\right\|_2^2\big)^{1/2}
 \qquad(x\in F)
\]
for some $c>0$. Both results hold for arbitrary $\mathcal M$.

By the support reduction in \cite[Remark~3.2]{HJX}, the left and
right supports of each $x\in L_p(\mathcal M)$ are {$\sigma$-finite}.
For each sequence $(x_n)\subset L_p(\mathcal M)$, the countable
join $e$ of its left and right supports is {$\sigma$-finite}.
The von Neumann algebra $e\mathcal M e$, with unit $e$,
is called a corner of $\mathcal M$. Moreover, \cite[Fact~1.4]{RX} gives
\begin{equation}\label{eq:sigma-finite-support-reduction}
 x_n=e x_n e\in eL_p(\mathcal M)e=L_p(e\mathcal M e)
 \qquad(n\in\mathbb N).
\end{equation}
A faithful density can then be chosen in this corner.

\section{Proofs of the main results}\label{sec:proofs}

\subsection{Positive tensors and orthogonal sums}\label{sec:orthogonal}

We begin with two estimates used in the proof of
Theorem~\ref{thm:general-all-p-upper}. The first follows directly
from Pisier's factorization formula \cite[Theorem~1.5]{PisierVector}.

\begin{lemma}\label{lem:positive-tensor-high-p}
Let $\mathcal N$ be a von Neumann algebra and $1\le s<\infty$.
For $m,N\in\mathbb N$, $a_1,\ldots,a_N\in(\mathbb{M}_m)_+$ and
$b_1,\ldots,b_N\in L_s(\mathcal N)_+$,
\[
 \left\|\sum_{k=1}^N a_k\otimes b_k\right\|_s
 \le \left\|\sum_{k=1}^N a_k\right\|_{S_s^m}
       \max_{1\le k\le N}\left\|b_k\right\|_s.
\]
\end{lemma}

\begin{proof}
Put $E\coloneqq L_s(\mathcal N)$ with its canonical operator-space
structure, and denote its matrix levels by $\mathbb{M}_l(E)$,
$l\in\mathbb N$.
We use the isometric identification
\[
 S_s^m[E]=L_s(\mathbb{M}_m\overline\otimes\mathcal N),
\]
valid for arbitrary $\mathcal N$; see
\cite[Section~6, immediately before Lemma~6.1]{RX}.
Pisier's factorization formula \cite[Theorem~1.5, p.~21]{PisierVector}
gives
\[
 \left\|uXv\right\|_{S_s^m[E]}
 \le\left\|u\right\|_{S_{2s}^m}\left\|X\right\|_{\mathbb{M}_m(E)}\left\|v\right\|_{S_{2s}^m}
 \qquad(u,v\in \mathbb{M}_m,\ X\in \mathbb{M}_m(E)).
\]

Set $A\coloneqq\sum_{k=1}^N a_k$. If $A=0$, then all $a_k=0$ and
there is nothing to prove. Otherwise, let $e$ be the support
projection of $A$ and let $A^{-1/2}$ denote the inverse of
$A^{1/2}$ on $e\mathbb C^m$, extended by zero on
$(\mathbf{1}-e)\mathbb C^m$. Define
\[
 c_k\coloneqq A^{-1/2}a_kA^{-1/2},\qquad
 W\coloneqq\begin{pmatrix}c_1^{1/2}&\cdots&c_N^{1/2}\end{pmatrix}
 \in \mathbb{M}_{m,Nm}(\mathbb C).
\]
Write $\left\|W\right\|$ for the operator norm of
$W:\mathbb C^{Nm}\to\mathbb C^m$.
Then $\sum_kc_k=e$, $a_k=A^{1/2}c_kA^{1/2}$ and
$WW^*=e$. Since $e$ is a nonzero projection,
\[
 \left\|W\right\|^2=\left\|WW^*\right\|_{B(\mathbb C^m)}
 =\left\|e\right\|_{B(\mathbb C^m)}=1.
\]
The matrix norm axioms give
\[
 \begin{aligned}
 X&:=\sum_{k=1}^N c_k\otimes b_k
   =W\operatorname{diag}(\mathbf{1}_{\mathbb M_m}\otimes b_1,\ldots,\mathbf{1}_{\mathbb M_m}\otimes b_N)W^*,\\
 \left\|X\right\|_{\mathbb{M}_m(E)}
 &\le\left\|W\right\|^2
       \left\|\operatorname{diag}(\mathbf{1}_{\mathbb M_m}\otimes b_1,\ldots,\mathbf{1}_{\mathbb M_m}\otimes b_N)
         \right\|_{\mathbb{M}_{Nm}(E)}
 =\max_k\left\|b_k\right\|_s.
 \end{aligned}
\]
Applying the factorization estimate with $u=v=A^{1/2}$ yields
\[
 \begin{aligned}
 \left\|\sum_k a_k\otimes b_k\right\|_{S_s^m[E]}
 &=\left\|A^{1/2}XA^{1/2}\right\|_{S_s^m[E]}\\
 &\le\left\|A^{1/2}\right\|_{S_{2s}^m}^{\,2}\max_k\left\|b_k\right\|_s
 =\left\|A\right\|_{S_s^m}\max_k\left\|b_k\right\|_s.
 \end{aligned}
\]
\end{proof}

{If $\mathcal N$ admits a faithful normal tracial state,
\cite[Theorem~3.1]{PisierOrth} gives the bound without an error term
for families satisfying
$\operatorname{Tr}(u_{i_1}^*u_{i_2}\cdots u_{i_{p-1}}^*u_{i_p})=0$
whenever $i_1,\ldots,i_p$ are pairwise distinct, with $3\pi p/2$
in place of $K_p$ below.} The estimate
with the error term $2^{1/p}\Delta_p(u)^{1/p}$ below follows
for arbitrary $\mathcal N$ from
Pisier's Banach-space tensor estimate \cite[Theorem~5.1]{PisierOrth}
and the noncommutative Khintchine inequalities
\cite[Theorem~6.2]{HJX}. We give the application below, with
$K_p\coloneqq4p\kappa_p$, where $\kappa_p$ is fixed in
\eqref{eq:khintchine-general}.
\Needspace{4\baselineskip}
\begin{lemma}
\label{lem:even-haagerup-upper}
{
Let $\mathcal N$ be a von Neumann algebra, let $p\in2\mathbb N$,
$p\ge4$, and let
$u=(u_1,\ldots,u_N)\in L_p(\mathcal N)^N$, $N\in\mathbb N$.
Set $K_p\coloneqq4p\kappa_p$, with $\kappa_p$ as in
\eqref{eq:khintchine-general}.
Define
\[
 \Delta_p(u)\coloneqq
 \sum_{\substack{1\le i_1,\ldots,i_p\le N\\
                 i_r\ne i_s\ (r\ne s)}}
 \left|\operatorname{Tr}(u_{i_1}^*u_{i_2}\cdots
                         u_{i_{p-1}}^*u_{i_p})\right|.
\]
Then
\[
 \left\|\sum_{k=1}^N u_k\right\|_p
 \le K_p\left\|\left(\sum_{k=1}^N(u_k^*u_k+u_ku_k^*)\right)^{1/2}\right\|_p
       +2^{1/p}\Delta_p(u)^{1/p}.
\]
In particular, if
\begin{equation}\label{eq:even-moment-vanishing}
 \operatorname{Tr}(u_{i_1}^*u_{i_2}\cdots
                 u_{i_{p-1}}^*u_{i_p})=0
\end{equation}
whenever $i_1,\ldots,i_p\in\{1,\ldots,N\}$ are pairwise distinct,
then $\Delta_p(u)=0$ and
\[
 \left\|\sum_{k=1}^N u_k\right\|_p
 \le K_p\left\|\left(\sum_{k=1}^N(u_k^*u_k+u_ku_k^*)\right)^{1/2}
           \right\|_p.
\]
}
\end{lemma}

\begin{proof}
Put $E\coloneqq L_p(\mathcal N)$ and
\[
 f\coloneqq\sum_{k=1}^N u_k,\qquad t\coloneqq\left\|f\right\|_p,\qquad
 s\coloneqq\left\|\left(\sum_{k=1}^N(u_k^*u_k+u_ku_k^*)\right)^{1/2}\right\|_p.
\]
Let $(\varepsilon_k)_{k=1}^N$ be independent Rademacher signs and set
\[
 R\coloneqq\left(\mathbb E\left\|\sum_{k=1}^N\varepsilon_k u_k\right\|_p^p
    \right)^{1/p}.
\]
By \eqref{eq:khintchine-general},
{
\begin{equation}\label{eq:rademacher-bound}
 R\le\kappa_p\max\left\{
 \left\|\left(\sum_k u_k^*u_k\right)^{1/2}\right\|_p,
 \left\|\left(\sum_k u_ku_k^*\right)^{1/2}\right\|_p
 \right\}\le\kappa_p s.
\end{equation}
}

On the algebraic tensor product $E^{\odot p}$, use the projective norm
\[
 \left\|z\right\|_\pi\coloneqq
 \inf_{\substack{L\in\mathbb N,\ x_{\ell,r}\in E\\
 z=\sum_{\ell=1}^Lx_{\ell,1}\otimes\cdots\otimes x_{\ell,p}}}
 \sum_{\ell=1}^L\prod_{r=1}^p\left\|x_{\ell,r}\right\|_p.
\]
Define the linear functional $\Lambda:E^{\odot p}\to\mathbb C$ by
\[
 \Lambda(x_1\otimes\cdots\otimes x_p)
 \coloneqq\operatorname{Tr}(x_1\cdots x_p).
\]
For every representation
$z=\sum_{\ell=1}^Lx_{\ell,1}\otimes\cdots\otimes x_{\ell,p}$,
{set $P_{\ell,j}\coloneqq x_{\ell,1}\cdots x_{\ell,j}$ for
$1\le j\le p$. Starting with $P_{\ell,1}=x_{\ell,1}$,
apply the noncommutative H\"older inequality
\cite[Section~1.2]{HJX} with exponents $p/(j-1)$ and $p$
for each $2\le j\le p$. This gives
\[
 P_{\ell,j}\in L_{p/j}(\mathcal N),\qquad
 \left\|P_{\ell,j}\right\|_{p/j}
 \le\prod_{r=1}^j\left\|x_{\ell,r}\right\|_p
 \quad(1\le j\le p).
\]
In particular, $P_{\ell,p}\in L_1(\mathcal N)$ and
$|\operatorname{Tr}(P_{\ell,p})|\le\left\|P_{\ell,p}\right\|_1$.
The triangle inequality therefore yields}
\[
 |\Lambda(z)|
 ={\left|\sum_{\ell=1}^L
       \operatorname{Tr}(P_{\ell,p})\right|}
 {\le\sum_{\ell=1}^L|\operatorname{Tr}(P_{\ell,p})|
       \le\sum_{\ell=1}^L\prod_{r=1}^p\left\|x_{\ell,r}\right\|_p.}
\]
Taking the infimum over these representations yields
\begin{equation}\label{eq:trace-projective-bound}
 |\Lambda(z)|\le\left\|z\right\|_\pi
 \qquad(z\in E^{\odot p}).
\end{equation}

For $1\le r\le p$, define
\[
 d_k(r)\coloneqq
 \begin{cases}
 u_k^*,&r\text{ odd},\\
 u_k,&r\text{ even},
 \end{cases}
 \qquad 1\le k\le N,\qquad
 F_r\coloneqq\sum_{k=1}^N d_k(r).
\]
Then
\[
 \left\|F_r\right\|_p=t,\qquad
 \left(\mathbb E\left\|\sum_{k=1}^N\varepsilon_kd_k(r)\right\|_p^p
 \right)^{1/p}=R
 \qquad(1\le r\le p).
\]
Set
\[
 \mathcal I_{p,N}
 \coloneqq\{(i_1,\ldots,i_p)\in\{1,\ldots,N\}^p:
       i_r\ne i_s\text{ whenever }r\ne s\}
\]
and define
\[
 U\coloneqq F_1\otimes\cdots\otimes F_p,\qquad
 V\coloneqq\sum_{(i_1,\ldots,i_p)\in\mathcal I_{p,N}}
       d_{i_1}(1)\otimes\cdots\otimes d_{i_p}(p).
\]
These definitions give
\begin{align*}
 \Lambda(U)
 &=\operatorname{Tr}(F_1\cdots F_p)
   =\operatorname{Tr}((f^*f)^{p/2})=t^p,\\
 |\Lambda(V)|
 &=\left|\sum_{(i_1,\ldots,i_p)\in\mathcal I_{p,N}}
       \operatorname{Tr}(u_{i_1}^*u_{i_2}\cdots
                        u_{i_{p-1}}^*u_{i_p})\right|
 \le\sum_{(i_1,\ldots,i_p)\in\mathcal I_{p,N}}
       \left|\operatorname{Tr}(u_{i_1}^*u_{i_2}\cdots
                        u_{i_{p-1}}^*u_{i_p})\right|
   =\Delta_p(u).
\end{align*}
Consequently, \eqref{eq:trace-projective-bound} implies
\begin{equation}\label{eq:even-trace-tensor-comparison}
 \begin{aligned}
 t^p&=|\Lambda(U)|
 \le|\Lambda(U-V)|+|\Lambda(V)|\\
 &\le\left\|U-V\right\|_\pi+\Delta_p(u).
 \end{aligned}
\end{equation}

Apply \cite[Theorem~5.1]{PisierOrth} with
$E_1=\cdots=E_p=E$ and the families $(d_k(r))_{k=1}^N$.
Since
$\#\{A\subset\{1,\ldots,p\}:|A|=j\}=\binom pj$, it follows that 
\begin{equation}\label{eq:pisier-tensor-subsets}
 \begin{aligned}
 \left\|U-V\right\|_\pi
 &\le\sum_{\substack{A\subset\{1,\ldots,p\}\\|A|\le p-2}}
       (p-|A|)!\,t^{|A|}R^{p-|A|}\\
 &=\sum_{j=0}^{p-2}\binom pj(p-j)!\,t^jR^{p-j}
 \overset{\eqref{eq:rademacher-bound}}{\le}
    \sum_{j=0}^{p-2}\binom pj(p-j)!\,t^j(\kappa_p s)^{p-j}.
 \end{aligned}
\end{equation}
Combining \eqref{eq:even-trace-tensor-comparison} and
\eqref{eq:pisier-tensor-subsets}, we obtain
\begin{equation}\label{eq:robust-even-polynomial}
 t^p\le\Delta_p(u)+
 \sum_{j=0}^{p-2}\binom pj(p-j)!\,t^j(\kappa_p s)^{p-j}.
\end{equation}

If $t\le4p\kappa_p s$, then
$t\le4p\kappa_p s+2^{1/p}\Delta_p(u)^{1/p}$.
If $t>4p\kappa_p s$, then $t>0$ and
\[
\frac1{t^p}\sum_{j=0}^{p-2}
       \binom pj(p-j)!\,t^j(\kappa_p s)^{p-j}
 =\sum_{k=2}^{p}\frac{p!}{(p-k)!}
            \left(\frac{\kappa_p s}{t}\right)^k
 \le\sum_{k=2}^{p}
            \left(\frac{p\kappa_p s}{t}\right)^k
 \le\sum_{k=2}^{\infty}4^{-k}=\frac1{12}.
\]
Thus \eqref{eq:robust-even-polynomial} gives
\[
 \frac{11}{12}t^p\le\Delta_p(u),\qquad
 t\le\left(\frac{12}{11}\right)^{1/p}\Delta_p(u)^{1/p}
 \le2^{1/p}\Delta_p(u)^{1/p}.
\]
In both cases,
\begin{equation}\label{eq:robust-even-upper}
 t\le4p\kappa_p s+2^{1/p}\Delta_p(u)^{1/p}.
\end{equation}
{Equation~\eqref{eq:robust-even-upper} proves the
general estimate. Under \eqref{eq:even-moment-vanishing}, every
summand in $\Delta_p(u)$ is zero, so the inequality without the
error term follows.}
\end{proof}

\subsection{Proof of Theorem~\ref{thm:general-all-p-upper}}\label{sec:analytic}

Throughout this subsection, write
\[
 \mathbb S\coloneqq\{z\in\mathbb C:0\le\operatorname{Re}z\le1\},
 \qquad
 \mathbb S^\circ\coloneqq\{z\in\mathbb C:0<\operatorname{Re}z<1\}.
\]
For $1\le q<\infty$, the notation $y_k\rightharpoonup y$ in
$L_q(\mathcal M)$ means weak convergence:
\[
 \operatorname{Tr}(a y_k)\longrightarrow\operatorname{Tr}(a y)
 \quad(a\in L_{q'}(\mathcal M)),\qquad
 \frac1q+\frac1{q'}=1.
\]

\begin{lemma}
\label{lem:haagerup-analytic-boundary-compactness}
Let $\mathcal M$ be a {$\sigma$-finite} von Neumann algebra, let
$d\in L_1(\mathcal M)_+$ be faithful with $\operatorname{Tr}d=1$,
and let $2<p<r<\infty$. Put
\[
 \delta\coloneqq\frac12-\frac1r,\qquad
 \theta\coloneqq\frac{1/2-1/p}{\delta}\in(0,1),\qquad
 r_0\coloneqq2,\quad r_1\coloneqq r.
\]
Suppose $(x_n)$ is bounded and weakly null in $L_p(\mathcal M)$,
and write $x_n=u_n h_n^{1/p}$, where $h_n\coloneqq|x_n|^p$.
Define, for $z\in\mathbb S$ and $t\in\mathbb R$,
\[
 F_n(z)\coloneqq d^{z\delta}u_n h_n^{1/2-z\delta},\qquad
 V_{n,j}(t)\coloneqq d^{it\delta}u_n h_n^{1/r_j-it\delta}
 \quad(j=0,1).
\]
Each $F_n:\mathbb S\to L_2(\mathcal M)$ is bounded
and analytic on $\mathbb S^\circ$, with
$F_n(j+it)=d^{j\delta}V_{n,j}(t)$.
The functions $F_n$ are weakly continuous on $\mathbb S$, and
the $V_{n,j}$ are bounded and strongly measurable. There are
integers $1\le n_1<n_2<\cdots$, a bounded, weakly continuous
$G:\mathbb S\to L_2(\mathcal M)$ analytic on $\mathbb S^\circ$,
and bounded, strongly measurable
$g_j:\mathbb R\to L_{r_j}(\mathcal M)$ such that
\[
 \begin{aligned}
 G(\theta)&=0,&
 F_{n_k}(z)&\rightharpoonup G(z)
     &&\text{in }L_2(\mathcal M)\quad(z\in\mathbb S),\\
 G(j+it)&=d^{j\delta}g_j(t),&
 V_{n_k,j}(t)&\rightharpoonup g_j(t)
     &&\text{in }L_{r_j}(\mathcal M)
       \quad(j=0,1,\ t\in\mathbb R).
 \end{aligned}
\]
\end{lemma}

\begin{proof}
Put
\[
 M\coloneqq\max\{1,\sup_n\operatorname{Tr}h_n\},\qquad
 a(z)\coloneqq\tfrac12+\delta z\quad(z\in\mathbb C).
\]
Set $\Omega\coloneqq\{z\in\mathbb C:
-1/(2\delta)<\operatorname{Re}z<1/(2\delta)\}$.
Since $0<\delta<1/2$, we have $\mathbb S\subset\Omega$ and
$0<\operatorname{Re}a(z)<1$ for $z\in\Omega$.
\par
\begin{samepage}
Put $k_n\coloneqq u_nh_nu_n^*\in L_1(\mathcal M)_+$.
The identity $k_n^{1-a(z)}u_n=u_nh_n^{1-a(z)}$ and
\cite[Lemma~9.20]{HiaiModular} give an
$L_1(\mathcal M)$-analytic\footnote{That is,
$B_n'(z)=\lim_{\zeta\to0}(B_n(z+\zeta)-B_n(z))/\zeta$
in the $L_1(\mathcal M)$ norm for every $z\in\Omega$.}
map $B_n:\Omega\to L_1(\mathcal M)$:
\[
 B_n(z):=d^{a(z)}k_n^{1-a(z)}u_n
          =d^{a(z)}u_nh_n^{1-a(z)}\quad(z\in\Omega).
\]
\end{samepage}
For $z\in\mathbb S$, we have $B_n(z)=d^{1/2}F_n(z)$.
For $t\in\mathbb R$, put
\[
 w_n(t)\coloneqq d^{it\delta}u_nh_n^{-it\delta},
\]
where $h_n^{-it\delta}$ is taken on $s(h_n)$.
Since $d,h_n\in L_1(\mathcal M)_+$ and $u_n\in\mathcal M$,
the dual-action scaling law \cite[(1.12)]{HJX} gives
\[
 \widehat\sigma_v(w_n(t))
 =(e^{-v}d)^{it\delta}u_n(e^{-v}h_n)^{-it\delta}
 =e^{-ivt\delta}e^{ivt\delta}w_n(t)=w_n(t)
       \qquad(v\in\mathbb R).
\]
Here $(e^{-v}h_n)^{-it\delta}=e^{ivt\delta}h_n^{-it\delta}$
on $s(h_n)$.
Thus $w_n(t)\in\mathcal R^{\widehat\sigma}=\mathcal M$
\cite[(1.3)]{HJX}. Since $d$ is faithful and
$u_n^*u_n=s(h_n)$,
\[
 w_n(t)^*w_n(t)
 =h_n^{it\delta}u_n^*d^{-it\delta}d^{it\delta}
       u_nh_n^{-it\delta}
 =h_n^{it\delta}s(h_n)h_n^{-it\delta}=s(h_n).
\]
Thus $|w_n(t)h_n^\beta|^2=h_n^\beta s(h_n)h_n^\beta
=h_n^{2\beta}$ for $\beta>0$.
Let $\alpha\ge0$, $\beta>0$, and $1/q=\alpha+\beta\le1$.
Put $p_\alpha\coloneqq1/\alpha$ if $\alpha>0$, and $p_0\coloneqq\infty$;
in particular, $\left\|d^0\right\|_{p_0}=1$.
H\"older's inequality \cite[Proposition~9.17]{HiaiModular}
gives
\[
 \begin{aligned}
 \left\|d^{\alpha+it\delta}u_nh_n^{\beta-it\delta}\right\|_q
 &=\left\|d^\alpha w_n(t)h_n^\beta\right\|_q\\
 &\le \left\|d^\alpha\right\|_{p_\alpha}
       \left\|w_n(t)h_n^\beta\right\|_{1/\beta}\\
 &=(\operatorname{Tr}d)^\alpha
   (\operatorname{Tr}h_n)^\beta\le M^\beta.
 \end{aligned}
\]
Taking
$(\alpha,\beta,q)=(s\delta,1/2-s\delta,2)$ for $F_n(s+it)$,
$(0,1/r_j,r_j)$ for $V_{n,j}(t)$, and
$(\operatorname{Re}a(z),1-\operatorname{Re}a(z),1)$ for
$B_n(z)$, with $t=\operatorname{Im}z$ in the last case, yields
\begin{equation}\label{eq:haagerup-basic-bounds}
 \begin{aligned}
 \left\|F_n(s+it)\right\|_2&\le M^{1/2-s\delta}
      &&(0\le s\le1),\\
 \left\|V_{n,j}(t)\right\|_{r_j}&\le M^{1/r_j}
      &&(j=0,1),\\
 \left\|B_n(z)\right\|_1&\le M^{1-\operatorname{Re}a(z)}
      &&(z\in\Omega).
 \end{aligned}
\end{equation}
For $b\in\mathcal M$, we have
\[
 \operatorname{Tr}\bigl((bd^{1/2})F_n(z)\bigr)
 =\operatorname{Tr}(bB_n(z))\qquad(z\in\mathbb S).
\]
Since $\mathcal M d^{1/2}$ is dense in $L_2(\mathcal M)$
\cite[Section~5, before Proposition~5.5]{HJX}, for each
$y\in L_2(\mathcal M)$ we may choose $b_m\in\mathcal M$ with
$b_md^{1/2}\to y$ in $L_2(\mathcal M)$. Since $M\ge1$, the first
estimate in \eqref{eq:haagerup-basic-bounds} gives
$\sup_{z\in\mathbb S}\left\|F_n(z)\right\|_2\le M^{1/2}$. Hence
\begin{equation}\label{eq:weak-pairing-uniform-convergence}
 \sup_{z\in\mathbb S}
 \bigl|\operatorname{Tr}(yF_n(z))
       -\operatorname{Tr}(b_mB_n(z))\bigr|
 \le M^{1/2}\left\|y-b_md^{1/2}\right\|_2\longrightarrow0.
\end{equation}
Since $|\operatorname{Tr}(b_mv)|\le\left\|b_m\right\|_\infty\left\|v\right\|_1$,
it follows from
\cite[Corollary~1.2.4]{MennickenMoller} that
$z\mapsto\operatorname{Tr}(b_mB_n(z))$ is analytic on $\Omega$.
By \eqref{eq:weak-pairing-uniform-convergence},
$z\mapsto\operatorname{Tr}(yF_n(z))$ is analytic on
$\mathbb S^\circ$ and continuous on $\mathbb S$ for every $y$.
By the weak-holomorphy criterion
\cite[Theorem~3.31(c)]{RudinFA}, $F_n$ is
$L_2$-analytic on $\mathbb S^\circ$.
The last estimate in \eqref{eq:haagerup-basic-bounds} gives
$\sup_{n,z\in\Omega}\left\|B_n(z)\right\|_1\le M$.
Set $\rho\coloneqq1/(2r\delta)$. Since
$\operatorname{dist}(\mathbb S,\partial\Omega)
=1/(2\delta)-1=1/(r\delta)=2\rho$,
we have $\overline{D(z,\rho)}\subset\Omega$ for $z\in\mathbb S$.
The $L_1$-valued Cauchy integral formula
\cite[Theorem~3.31(b)]{RudinFA} gives
\[
 \begin{aligned}
 B_n'(z)&=\frac{1}{2\pi i}\int_{|\zeta-z|=\rho}
            \frac{B_n(\zeta)}{(\zeta-z)^2}\,d\zeta,\\
 \left\|B_n'(z)\right\|_1
 &\le\frac{1}{2\pi}(2\pi\rho)\frac{M}{\rho^2}
   =\frac M\rho \qquad(z\in\mathbb S).
 \end{aligned}
\]
Consequently, with $C\coloneqq M/\rho=2r\delta M$,
\begin{equation}\label{eq:haagerup-common-L1-lipschitz}
 \left\|B_n(z)-B_n(w)\right\|_1\le C|z-w|
 \qquad(z,w\in\mathbb S).
\end{equation}
For $j=0,1$, set $\alpha_j\coloneqq1/2+j\delta=1-1/r_j$ and define
\[
 J_j:L_{r_j}(\mathcal M)\to L_1(\mathcal M),
 \qquad J_jy\coloneqq d^{\alpha_j}y.
\]
By H\"older's inequality \cite[Section~1.2]{HJX} and
$\operatorname{Tr}d=1$,
\[
 \left\|J_jy\right\|_1\le
 \left\|d^{\alpha_j}\right\|_{1/\alpha_j}\left\|y\right\|_{r_j}
 =(\operatorname{Tr}d)^{\alpha_j}\left\|y\right\|_{r_j}
 =\left\|y\right\|_{r_j}.
\]
Since $s(d)=\mathbf{1}$, each $J_j$ is injective.
Moreover, $J_0F_n(z)=B_n(z)$ and
$J_jV_{n,j}(t)=B_n(j+it)$.
Write $B_{L_{r_j}(\mathcal M)}$ for the closed unit ball.
For $R>0$, reflexivity makes $RB_{L_{r_j}(\mathcal M)}$
weakly compact. Boundedness of $J_j$ makes
$J_j(RB_{L_{r_j}(\mathcal M)})$ weakly compact in $L_1$.
The restriction of $J_j$ to $RB_{L_{r_j}(\mathcal M)}$ is a
weakly continuous bijection onto the Hausdorff space
$J_j(RB_{L_{r_j}(\mathcal M)})$. Hence
$J_j^{-1}:J_j(RB_{L_{r_j}(\mathcal M)})
\to RB_{L_{r_j}(\mathcal M)}$ is weakly continuous
\cite[Corollary~0.10]{Teschl}.

Choose a countable dense set $D\subset\mathbb S$.
By \eqref{eq:haagerup-basic-bounds},
$F_n(z)\in M^{1/2}B_{L_2(\mathcal M)}$ for $z\in\mathbb S$.
For each $z\in D$, weak sequential compactness of
$M^{1/2}B_{L_2(\mathcal M)}$ and boundedness of $J_0$ give a
weakly convergent subsequence of $(B_n(z))_n$ in $L_1$.
Since $D$ is countable, a diagonal subsequence $(B_{n_k})_k$
converges weakly at every $z\in D$.
Estimate \eqref{eq:haagerup-common-L1-lipschitz}
extends the convergence to $\mathbb S$ and gives
\[
 B_{n_k}(z)\rightharpoonup B(z)\quad\text{in }L_1
 \quad(z\in\mathbb S),\qquad
 \left\|B(z)-B(w)\right\|_1\le C|z-w|.
\]
Recall that $\delta=1/2-1/r$, so $1/2-\delta=1/r$.
For $z=s+it\in\mathbb S$, \eqref{eq:haagerup-basic-bounds} gives
$\left\|F_{n_k}(z)\right\|_2\le M^{1/2-s\delta}\le M^{1/2}$ and
$\left\|V_{n_k,1}(t)\right\|_r\le M^{1/r}$. Since
\[
 \begin{aligned}
 B_{n_k}(z)
 &=d^{1/2}\bigl(d^{z\delta}u_{n_k}h_{n_k}^{1/2-z\delta}\bigr)
   =J_0F_{n_k}(z)
 \in J_0(M^{1/2}B_{L_2(\mathcal M)})
   \qquad(z\in\mathbb S),\\
 B_{n_k}(1+it)
 &=d^{1/2+\delta}
   \bigl(d^{it\delta}u_{n_k}h_{n_k}^{1/r-it\delta}\bigr)
   =J_1V_{n_k,1}(t)
 \in J_1(M^{1/r}B_{L_r(\mathcal M)})
   \qquad(t\in\mathbb R),
 \end{aligned}
\]
it follows from the weak closedness of
$J_j(RB_{L_{r_j}(\mathcal M)})$ that
$B(z)\in J_0(M^{1/2}B_{L_2})$ and
$B(1+it)\in J_1(M^{1/r}B_{L_r})$. Set
\[
 G(z)\coloneqq J_0^{-1}B(z),\qquad
 g_0(t)\coloneqq G(it),\qquad g_1(t)\coloneqq J_1^{-1}B(1+it).
\]
Since $J_0(d^\delta g_1(t))=J_1g_1(t)
=B(1+it)=J_0G(1+it)$ and $J_0$ is injective,
\[
 G(1+it)=d^\delta g_1(t)\qquad(t\in\mathbb R).
\]
The inverse $J_0^{-1}$ is weakly continuous on
$J_0(M^{1/2}B_{L_2})$, and $J_1^{-1}$ is weakly continuous on
$J_1(M^{1/r}B_{L_r})$.
Applying $J_0^{-1}$ to $B_{n_k}(z)\rightharpoonup B(z)$ and
$J_1^{-1}$ at $z=1+it$ gives
\begin{equation}\label{eq:haagerup-Fn-weak-limit}
 \begin{aligned}
 F_{n_k}(z)=J_0^{-1}B_{n_k}(z)&\rightharpoonup G(z)
   &&\text{in }L_2\quad(z\in\mathbb S),\\
 V_{n_k,j}(t)=J_j^{-1}B_{n_k}(j+it)&\rightharpoonup g_j(t)
   &&\text{in }L_{r_j}\quad(j=0,1,\ t\in\mathbb R).
 \end{aligned}
\end{equation}
\Needspace{5\baselineskip}
Let $z_m\to z$ in $\mathbb S$ and $t_m\to t$ in $\mathbb R$.
The estimate \eqref{eq:haagerup-common-L1-lipschitz} and the
bound for $B$ give
\[
 \left\|B_n(z_m)-B_n(z)\right\|_1+\left\|B(z_m)-B(z)\right\|_1
 \le 2C|z_m-z|\longrightarrow0.
\]
Apply the weakly continuous $J_0^{-1}$ on
$J_0(M^{1/2}B_{L_2})$ and $J_j^{-1}$ on
$J_j(M^{1/r_j}B_{L_{r_j}})$, taking $z_m=j+it_m$ for the
boundary values. Then, for $j=0,1$,
\begin{equation}\label{eq:boundary-weak-continuity}
 \begin{aligned}
 F_n(z_m)&\rightharpoonup F_n(z),&
 G(z_m)&\rightharpoonup G(z)&&\text{in }L_2,\\
 V_{n,j}(t_m)&\rightharpoonup V_{n,j}(t),&
 g_j(t_m)&\rightharpoonup g_j(t)&&\text{in }L_{r_j}.
 \end{aligned}
\end{equation}
We claim that $V_{n,j}$ and $g_j$ are strongly measurable.
Indeed, for $f=V_{n,j}$ or $f=g_j$, set
\[
 E_f:=\overline{\operatorname{span}}^{\left\|\cdot\right\|}
     \{f(q):q\in\mathbb Q\}.
\]
For $x\notin E_f$, the closed-subspace Hahn--Banach theorem
\cite[Theorem~3.5]{RudinFA} supplies $\ell\in L_{r_j}^*$ with
$\ell|_{E_f}=0$ and $\ell(x)=1$. Consequently,
$E_f=\bigcap_{\ell|_{E_f}=0}\ker\ell$ is weakly closed:
each $\ker\ell=\ell^{-1}(\{0\})$ is weakly closed because
$\ell$ is weakly continuous \cite[Section~3.11]{RudinFA}.
For each $t\in\mathbb R$, choose $q_m\in\mathbb Q$ with
$q_m\to t$. By \eqref{eq:boundary-weak-continuity},
$f(q_m)\rightharpoonup f(t)$; since $f(q_m)\in E_f$, we obtain
$f(t)\in E_f$. Equation~\eqref{eq:boundary-weak-continuity}
also gives $\ell(f(t_m))\to\ell(f(t))$ whenever $t_m\to t$
and $\ell\in L_{r_j}^*$. Thus $\ell\circ f$ is continuous and
therefore Borel measurable \cite[Section~1.11]{RudinRCA}, hence
Lebesgue measurable. Since $E_f$ is separable
(as $\mathbb Q$ is countable) and $f$ is weakly measurable,
it follows from Pettis' theorem \cite[Theorem~II.1.2]{DU}
that $f$ is strongly measurable.

Fix $y\in L_2$ and put
$\varphi_{k,y}(z)\coloneqq\operatorname{Tr}(yF_{n_k}(z))$.
Since $|\operatorname{Tr}(yv)|\le\left\|y\right\|_2\left\|v\right\|_2$ for $v\in L_2$,
it follows from the $L_2$-analyticity of $F_{n_k}$ and
\cite[Corollary~1.2.4]{MennickenMoller} that
$\varphi_{k,y}$ is analytic on $\mathbb S^\circ$.
Equations~\eqref{eq:haagerup-basic-bounds} and
\eqref{eq:haagerup-Fn-weak-limit} give, respectively,
\[
 |\varphi_{k,y}(z)|\le M^{1/2}\left\|y\right\|_2,
 \qquad
 \varphi_{k,y}(z)\longrightarrow\operatorname{Tr}(yG(z))
 \quad(z\in\mathbb S^\circ).
\]
By Vitali's theorem \cite[Theorem~13.5]{StollComplex},
$z\mapsto\operatorname{Tr}(yG(z))$ is analytic. Moreover,
$\left\|G(z)\right\|_2\le M^{1/2}$ by weak lower semicontinuity.
Since $y$ was arbitrarily taken, the weak-holomorphy criterion
\cite[Theorem~3.31(c)]{RudinFA} shows that $G$ is
$L_2$-analytic on $\mathbb S^\circ$.

Finally, $1/2-\theta\delta=1/p$, so the definition of $F_n$
and the polar decomposition $x_n=u_nh_n^{1/p}$ give
\[
 F_{n_k}(\theta)
 =d^{\theta\delta}u_{n_k}h_{n_k}^{1/2-\theta\delta}
 =d^{\theta\delta}x_{n_k}.
\]
Also $1/p'=1/2+\theta\delta$, and H\"older's inequality gives
\[
 \left\|yd^{\theta\delta}\right\|_{p'}
 \le\left\|y\right\|_2\left\|d^{\theta\delta}\right\|_{1/(\theta\delta)}
 =\left\|y\right\|_2\qquad(y\in L_2).
\]
Since $x_{n_k}\rightharpoonup0$ in $L_p$ and
$yd^{\theta\delta}\in L_{p'}$, it follows that
\[
 \operatorname{Tr}\bigl(yF_{n_k}(\theta)\bigr)
 =\operatorname{Tr}\bigl((yd^{\theta\delta})x_{n_k}\bigr)
 \longrightarrow0\qquad(y\in L_2).
\]
By \eqref{eq:haagerup-Fn-weak-limit} at $z=\theta$,
$\operatorname{Tr}(yG(\theta))=0$ for every $y\in L_2$,
i.e., $G(\theta)=0$.
\end{proof}

\begin{lemma}\label{lem:weighted-boundary-interpolation}
Let $\mathcal N$ be a von Neumann algebra and let
$d\in L_1(\mathcal N)_+$ be faithful. Fix $2<r<\infty$ and
$0<\theta<1$. Set
$\mathbb S\coloneqq\{z\in\mathbb C:0\le\operatorname{Re}z\le1\}$.
Write $L_q\coloneqq L_q(\mathcal N)$ and set
\[
\delta\coloneqq\frac12-\frac1r,\qquad
 \frac1p\coloneqq\frac12-\theta\delta,\qquad
(r_0,r_1)\coloneqq(2,r),\qquad
 (\lambda_0,\lambda_1)\coloneqq(1-\theta,\theta).
\]
For $j=0,1$, define $P_{\theta,j}$ and $\nu_j$ by
\begin{equation}\label{eq:strip-boundary-measures}
 \begin{gathered}
 P_{\theta,j}(t)
 \coloneqq\frac{\sin(\pi\theta)}
 {2\bigl(\cosh(\pi t)-(-1)^j\cos(\pi\theta)\bigr)},
 \qquad d\nu_j(t)\coloneqq\frac{P_{\theta,j}(t)}{\lambda_j}\,dt,\\
 \int_{\mathbb R}P_{\theta,j}(t)\,dt=\lambda_j,
 \qquad \int_{\mathbb R}d\nu_j(t)=1
 \quad(j=0,1).
 \end{gathered}
\end{equation}

Here $dt$ denotes Lebesgue measure in the real
variable $t$, while $d$ is the fixed element of $L_1(\mathcal N)$.
Suppose
$H:\mathbb S^\circ\to L_2$ is bounded and analytic, and
$H_j:\mathbb R\to L_{r_j}$ are bounded and strongly measurable,
with weak $L_2$ limits
\[
 H(s+it)\rightharpoonup d^{j\delta}H_j(t)
 \quad(s\to j,\ 0<s<1;\ j=0,1,\ t\in\mathbb R).
\]
Then there is a unique $y\in L_p$ such that
$H(\theta)=d^{\theta\delta}y$, and
\begin{equation}\label{eq:integrated-kosaki}
 \left\|y\right\|_p
 \le\left(\int_{\mathbb R}\left\|H_0(t)\right\|_2^2\,d\nu_0(t)\right)^{(1-\theta)/2}
   \left(\int_{\mathbb R}\left\|H_1(t)\right\|_r^r\,d\nu_1(t)\right)^{\theta/r}.
\end{equation}
\end{lemma}

\begin{proof}
Put
\[
 X_0\coloneqq L_2,\qquad X_1\coloneqq d^\delta L_r,
 \qquad X_\theta\coloneqq[X_0,X_1]_\theta.
\]
Equip $d\mathcal N$ with $\left\|da\right\|_{d\mathcal N}\coloneqq\left\|a\right\|_\infty$.
Kosaki's identity \cite[Theorem~9.1]{Kosaki},
applied after normalizing $d$ if necessary, is
\[
 [d\mathcal N,L_1(\mathcal N)]_{1/q}
 =d^{1-1/q}L_q(\mathcal N)\qquad(1<q<\infty).
\]
Reiteration \cite[Theorem~4.6.1]{BL}
and multiplication by $d^{-1/2}$ give
\begin{equation}\label{eq:weighted-kosaki-interpolation}
 \begin{aligned}
 [L_2,d^\delta L_r]_\theta
 &=d^{-1/2}[d^{1/2}L_2,d^{1-1/r}L_r]_\theta
 =d^{-1/2}d^{1-1/p}L_p=d^{\theta\delta}L_p,\\
 \left\|d^{\theta\delta}y\right\|_{[L_2,d^\delta L_r]_\theta}
 &=\left\|y\right\|_p\qquad(y\in L_p).
 \end{aligned}
\end{equation}
Following Peetre's argument
\cite[Section~4, proof of the theorem]{PeetreHinf},
for $\varepsilon>0$ define the analytic function $H_\varepsilon$
and its continuous boundary functions $H_{j,\varepsilon}$ by
\[
 \begin{aligned}
 H_\varepsilon(z)
 &\coloneqq\frac1\varepsilon\int_0^\varepsilon H(z+iu)\,du,\\
 H_\varepsilon(j+it)
 &=d^{j\delta}H_{j,\varepsilon}(t),&
 H_{j,\varepsilon}(t)
 &\coloneqq\frac1\varepsilon\int_0^\varepsilon H_j(t+u)\,du.
 \end{aligned}
\]
For $h\in\mathbb R$, the definition of
$H_{j,\varepsilon}$ gives
\[
 \begin{aligned}
 H_{j,\varepsilon}(t+h)-H_{j,\varepsilon}(t)
 &=\frac1\varepsilon\left(
   \int_{t+\varepsilon}^{t+\varepsilon+h}H_j(v)\,dv
   -\int_t^{t+h}H_j(v)\,dv\right),\\
 \left\|H_{j,\varepsilon}(t+h)-H_{j,\varepsilon}(t)\right\|_{r_j}
 &\le \frac{2|h|}{\varepsilon}
       \sup_{u\in\mathbb R}\left\|H_j(u)\right\|_{r_j}.
 \end{aligned}
\]
Thus $H_{j,\varepsilon}$ is norm continuous. Peetre's lemma
\cite[Section~2, Lemma]{PeetreHinf}
then gives an $L_2$-continuous extension of $H_\varepsilon$
to the closed strip
$\mathbb S=\{z\in\mathbb C:
0\le\operatorname{Re}z\le1\}$.
For $j=0,1$, the function $P_{\theta,j}$ defined in
\eqref{eq:strip-boundary-measures} is the Poisson kernel\footnote{
For a bounded harmonic function $u$ on $\mathbb S^\circ$
that is continuous on $\mathbb S$, the Poisson formula reads
\[
 u(\theta)=\sum_{k=0}^1\int_{\mathbb R}
    u(k+it)P_{\theta,k}(t)\,dt.
\]
Thus $P_{\theta,j}(t)\,dt$ is harmonic measure at $\theta$
on the boundary line $j+i\mathbb R$.}
at $\theta$ on the boundary line $\{j+it:t\in\mathbb R\}$ of $\mathbb S$;
see \cite[Chapter~8, Section~1, p.~216]{SteinShakarchiComplex}
after rotating the horizontal strip.
Moreover, $\nu_j(\mathbb R)=1$ for $j=0,1$.\footnote{
Let $u_0(z)\coloneqq1-\operatorname{Re}z$ and
$u_1(z)\coloneqq\operatorname{Re}z$. Since
$u_k(j+it)=\mathbf 1_{\{j=k\}}$, the Poisson formula gives
\[
 \begin{aligned}
 \int_{\mathbb R}P_{\theta,k}(t)\,dt
 &=\sum_{j=0}^1\int_{\mathbb R}
    u_k(j+it)P_{\theta,j}(t)\,dt\\
 &=u_k(\theta)
 =\begin{cases}
   1-\theta,&k=0,\\
   \theta,&k=1.
  \end{cases}
 \end{aligned}
\]
Consequently,
$\nu_k(\mathbb R)
=\lambda_k^{-1}\int_{\mathbb R}P_{\theta,k}(t)\,dt=1$.}
Calder\'on's inequality
\cite[Section~4, formula~(1)]{PeetreHinf} applies directly to
$H_\varepsilon$. Together with Jensen's inequality
\cite[Theorem~3.3]{RudinRCA}, it gives
\[
 \log\left\|H_\varepsilon(\theta)\right\|_{X_\theta}
 \le\sum_{j=0}^1\lambda_j\int_{\mathbb R}
       \log\left\|H_{j,\varepsilon}(t)\right\|_{r_j}\,d\nu_j(t)
 \le\sum_{j=0}^1\frac{\lambda_j}{r_j}
       \log\int_{\mathbb R}
       \left\|H_{j,\varepsilon}(t)\right\|_{r_j}^{r_j}\,d\nu_j(t).
\]
Exponentiating the displayed inequality yields
\begin{equation}\label{eq:averaged-interpolation-bound}
 \left\|H_\varepsilon(\theta)\right\|_{X_\theta}
 \le\prod_{j=0}^1
 \left(\int_{\mathbb R}
   \left\|H_{j,\varepsilon}(t)\right\|_{r_j}^{r_j}
       \,d\nu_j(t)\right)^{\lambda_j/r_j}.
\end{equation}
By Bochner-valued Lebesgue differentiation
\cite[Theorem~2.3.4]{HNVW1},
$H_{j,\varepsilon}(t)\to H_j(t)$ for Lebesgue-almost every $t$.
Formula~\eqref{eq:strip-boundary-measures} shows that the same
holds for $\nu_j$-almost every $t$; boundedness of $H_j$
and dominated convergence then give
\begin{equation}\label{eq:boundary-average-convergence}
 \int_{\mathbb R}\left\|H_{j,\varepsilon}(t)-H_j(t)\right\|_{r_j}^{r_j}
       \,d\nu_j(t)\longrightarrow0\qquad(j=0,1).
\end{equation}
To apply \eqref{eq:averaged-interpolation-bound}
to a difference, set
$K_{\varepsilon,\eta}\coloneqq H_\varepsilon-H_\eta$. Then
\[
 K_{\varepsilon,\eta}(j+it)
 =d^{j\delta}\bigl(H_{j,\varepsilon}(t)-H_{j,\eta}(t)\bigr)
 \qquad(j=0,1).
\]
Repeating the derivation of
\eqref{eq:averaged-interpolation-bound} for
$K_{\varepsilon,\eta}$ gives
\[
 \left\|H_\varepsilon(\theta)-H_\eta(\theta)\right\|_{X_\theta}
 \le\prod_{j=0}^1
 \left(\int_{\mathbb R}
   \left\|H_{j,\varepsilon}(t)-H_{j,\eta}(t)\right\|_{r_j}^{r_j}
       \,d\nu_j(t)\right)^{\lambda_j/r_j}\longrightarrow0.
\]
Thus $H_\varepsilon(\theta)$ converges in
$X_\theta$, which is complete. Since
$H_\varepsilon(\theta)\to H(\theta)$ in $L_2$ and
H\"older's inequality gives
$\left\|d^{\theta\delta}v\right\|_2\le
(\operatorname{Tr}d)^{\theta\delta}\left\|v\right\|_p$, the
$X_\theta$-limit is
$H(\theta)=d^{\theta\delta}y$ for a unique $y\in L_p$.
By \eqref{eq:weighted-kosaki-interpolation},
\eqref{eq:averaged-interpolation-bound} and
\eqref{eq:boundary-average-convergence},
\[
 \begin{aligned}
 \left\|y\right\|_p
 &=\left\|d^{\theta\delta}y\right\|_{X_\theta}
   =\left\|H(\theta)\right\|_{X_\theta}\\
 &=\lim_{\varepsilon\downarrow0}
   \left\|H_\varepsilon(\theta)\right\|_{X_\theta}\\
 &\le\prod_{j=0}^1\left(
   \int_{\mathbb R}\left\|H_j(t)\right\|_{r_j}^{r_j}
   \,d\nu_j(t)\right)^{\lambda_j/r_j}\\
 &=\left(\int_{\mathbb R}\left\|H_0(t)\right\|_2^2\,d\nu_0(t)\right)^{(1-\theta)/2}
   \left(\int_{\mathbb R}\left\|H_1(t)\right\|_r^r\,d\nu_1(t)\right)^{\theta/r}.
 \end{aligned}
\]
The final equality uses
$\lambda_0/r_0=(1-\theta)/2$ and
$\lambda_1/r_1=\theta/r$.
\end{proof}

\begin{proof}[Proof of Theorem~\ref{thm:general-all-p-upper}]
If $\sup_n\left\|x_n\right\|_p=0$, the conclusion is immediate.
By homogeneity, without loss of generality, we may assume that
$\sup_n\left\|x_n\right\|_p\le1$.
Let $e$ be the projection in \eqref{eq:sigma-finite-support-reduction}.
Then $e\mathcal M e$ is $\sigma$-finite and
$x_n\in L_p(e\mathcal M e)$ for all $n$.
We replace $\mathcal M$ by the von Neumann algebra $e\mathcal M e$.
Choose a faithful
$d\in L_1(\mathcal M)_+$ with $\operatorname{Tr}d=1$.
Set
\begin{equation}\label{eq:theorem-interpolation-parameters}
 r\coloneqq2\lfloor p/2\rfloor+2,\qquad
 \delta\coloneqq\frac12-\frac1r,\qquad
 \theta\coloneqq\frac{1/2-1/p}{\delta}\in(0,1).
\end{equation}
Put $r_0\coloneqq2$ and $r_1\coloneqq r$.

\medskip
\noindent\textbf{Step 1.} We construct analytic functions $(K_n)$
with boundary values $K_n(j+it)=d^{j\delta}v_{n,j}(t)$, where
$(v_{n,j}(t))_n$ is uniformly bounded and weakly null in
$L_{r_j}(\mathcal M)$ for every $t\in\mathbb R$ and $j=0,1$.

Pass to the subsequence given by
Lemma~\ref{lem:haagerup-analytic-boundary-compactness}.
Relabel $x_{n_k}$, $F_{n_k}$, and $V_{n_k,j}$ as
$x_k$, $F_k$, and $V_{k,j}$, respectively, for $j=0,1$.
Write $K_n\coloneqq F_n-G$
($G$ is given by
Lemma~\ref{lem:haagerup-analytic-boundary-compactness}), and set
\[
 v_{n,j}(t)\coloneqq V_{n,j}(t)-g_j(t)\quad(j=0,1).
\]
Here $g_0,g_1$ are also supplied by
Lemma~\ref{lem:haagerup-analytic-boundary-compactness}. The boundary
identities in Lemma~\ref{lem:haagerup-analytic-boundary-compactness} give
\begin{equation}\label{eq:theorem-K-boundary-values}
 \begin{aligned}
 K_n(\theta)&=F_n(\theta)-G(\theta)=d^{\theta\delta}x_n,\\
 K_n(j+it)&=F_n(j+it)-G(j+it)\\
 &=d^{j\delta}\bigl(V_{n,j}(t)-g_j(t)\bigr)
  =d^{j\delta}v_{n,j}(t)\quad(j=0,1).
 \end{aligned}
\end{equation}
For $j\in\{0,1\}$, each map $t\mapsto v_{n,j}(t)$ is strongly
measurable (Lemma~\ref{lem:haagerup-analytic-boundary-compactness}).
Since $\operatorname{Tr}|x_n|^p\le1$, it follows that the constant
$M=\max\{1,\sup_n\operatorname{Tr}|x_n|^p\}$ in
\eqref{eq:haagerup-basic-bounds} equals $1$. That estimate and the
weak convergence in
\eqref{eq:haagerup-Fn-weak-limit} give, by weak lower
semicontinuity of the norm,
\[
 \left\|V_{n,j}(t)\right\|_{r_j}\le1,\qquad
 \left\|g_j(t)\right\|_{r_j}\le1,\qquad
 \left\|v_{n,j}(t)\right\|_{r_j}
 \le\left\|V_{n,j}(t)\right\|_{r_j}+\left\|g_j(t)\right\|_{r_j}\le2.
\]
Subtracting $g_j(t)$ from
\eqref{eq:haagerup-Fn-weak-limit}, we obtain
\begin{equation}\label{eq:theorem-boundary-limits}
 v_{n,j}(t)\xrightarrow[n\to\infty]{w}0
 \quad\text{in }L_{r_j}(\mathcal M),\qquad
 \sup_{n\in\mathbb N,\,t\in\mathbb R}\left\|v_{n,j}(t)\right\|_{r_j}\le2
 \quad(j\in\{0,1\}).
\end{equation}
The weak convergence in
\eqref{eq:theorem-boundary-limits} holds for every $t\in\mathbb R$.

Use the probability measures $\nu_0,\nu_1$ from
\eqref{eq:strip-boundary-measures}. For $q\in\{2,r\}$
set $j_2\coloneqq0$, $j_r\coloneqq1$, and
define
\begin{equation}\label{eq:distinct-q-tuples}
 \mathcal I_q\coloneqq
 \bigl\{(i_1,\ldots,i_q)\in\mathbb N^q:
 i_a\ne i_b\ \text{whenever }a\ne b\bigr\}.
\end{equation}
For $\mathbf i\in\mathcal I_q$, define
\begin{equation}\label{eq:Phi-q-definition}
 \Phi_q(\mathbf i)
 \coloneqq\int_{\mathbb R}\left|
   \operatorname{Tr}\big(
   v_{i_1,j_q}(t)^*v_{i_2,j_q}(t)\cdots
   v_{i_{q-1},j_q}(t)^*v_{i_q,j_q}(t)\big)
   \right|\,d\nu_{j_q}(t).
\end{equation}
\medskip
\noindent\textbf{Step 2.} We choose a subsequence $(x_{n_k})$
of $(x_n)$ such that
\begin{equation}\label{eq:general-analytic-moment-selection}
 \sum_{(i_1,\ldots,i_q)\in\mathcal I_q}
 \Phi_q(n_{i_1},\ldots,n_{i_q})\le1\qquad(q=2,r).
\end{equation}

Fix distinct $i_1,\ldots,i_{q-1}$, a position
$1\le\ell\le q$, and $t\in\mathbb R$. In the integrand of
\eqref{eq:Phi-q-definition}, set
\[
 \begin{gathered}
 \mathbf h(n)\coloneqq(i_1,\ldots,i_{\ell-1},n,i_\ell,\ldots,i_{q-1}),
 \qquad
 w_s(t)\coloneqq
 \begin{cases}
 v_{h_s(n),j_q}(t)^*,&s\text{ odd},\\
 v_{h_s(n),j_q}(t),&s\text{ even},
 \end{cases}\\
 C_\ell(t)\coloneqq w_{\ell+1}(t)\cdots w_q(t)
 w_1(t)\cdots w_{\ell-1}(t)
 \in L_{q/(q-1)}(\mathcal M).
 \end{gathered}
\]
The product $C_\ell(t)$ is independent of $n$.
Trace cyclicity and \eqref{eq:theorem-boundary-limits} now give
\[
 \operatorname{Tr}(w_1(t)\cdots w_q(t))
 =\begin{cases}
 \overline{\operatorname{Tr}(C_\ell(t)^*v_{n,j_q}(t))},
       &\ell\text{ odd},\\
 \operatorname{Tr}(C_\ell(t)v_{n,j_q}(t)),
       &\ell\text{ even}
 \end{cases}
 \xrightarrow[n\to\infty]{}0.
\]
H\"older's inequality and
\eqref{eq:theorem-boundary-limits} give the integrable bound
\[
 \bigl|\operatorname{Tr}(w_1(t)\cdots w_q(t))\bigr|
 \le\prod_{s=1}^q\left\|w_s(t)\right\|_q\le2^q.
\]
Since $\nu_{j_q}(\mathbb R)=1$ by
\eqref{eq:strip-boundary-measures}, it follows from dominated convergence that
\[
 \Phi_q(i_1,\ldots,i_{\ell-1},n,i_{\ell},\ldots,i_{q-1})
 \longrightarrow0\quad(n\to\infty,\ 1\le \ell\le q).
\]
Set $n_1\coloneqq1$. Given
$n_1<\cdots<n_{k-1}$ and $n>n_{k-1}$, define
\[
 \begin{aligned}
 \mathcal T_{q,k}(n)&\coloneqq
 \bigl\{\mathbf i\in\mathcal I_q\cap
 \{n_1,\ldots,n_{k-1},n\}^q:
 \max\{i_1,\ldots,i_q\}=n\bigr\},\\
 |\mathcal T_{q,k}(n)|&\le q(k-1)^{q-1}<\infty,\qquad
 \lim_{n\to\infty}\sum_{\mathbf i\in\mathcal T_{q,k}(n)}
 \Phi_q(\mathbf i)=0\quad(q=2,r).
 \end{aligned}
\]
For $k<q$, $\mathcal T_{q,k}(n)=\varnothing$.
Choose $n_k>n_{k-1}$ recursively, simultaneously for
$q=2,r$, so that
\begin{equation}\label{eq:moment-recursion}
 \sum_{\mathbf i\in\mathcal T_{q,k}(n_k)}
 \Phi_q(\mathbf i)\le2^{-k}\qquad(q=2,r,\ k\ge2).
\end{equation}
Relabel $v_{n_k,j}$, $K_{n_k}$, and $x_{n_k}$ as
$v_{k,j}$, $K_k$, and $x_k$.
For the relabeled sequence, $\Phi_q$ is computed using
\eqref{eq:Phi-q-definition}; thus summing
\eqref{eq:moment-recursion} by the maximal index of each tuple
in \eqref{eq:distinct-q-tuples} gives
\[
 \begin{aligned}
 \sum_{\mathbf i\in\mathcal I_q}\Phi_q(\mathbf i)
 &=\sum_{k=q}^\infty
 \sum_{\substack{\mathbf i\in\{1,\ldots,k\}^q\\
                  i_1,\ldots,i_q\text{ distinct},\
                  \max\{i_1,\ldots,i_q\}=k}}
 \Phi_q(\mathbf i)\\
 &\le\sum_{k=q}^\infty2^{-k}=2^{1-q}\le\frac12\le1
 \qquad(q=2,r).
 \end{aligned}
\]
The chosen sequence of indices $(n_k)$ depends only on the
boundary functions and is fixed before choosing $m,N$ or $(a_k)$.

Fix $m,N\in\mathbb N$ and $a_1,\ldots,a_N\in \mathbb{M}_m$. Let
$\varepsilon_1,\ldots,\varepsilon_N$ be independent Rademacher variables on
$\Omega\coloneqq\{-1,1\}^N$ with uniform probability measure $\mathbb P$.
For a matrix-valued function $f$ on $\Omega$, write
$\mathbb E f\coloneqq 2^{-N}\sum_{\omega\in\Omega}f(\omega)$. Put
\begin{equation}\label{eq:coefficient-normalizing-factor}
 L\coloneqq\left\|\sum_{k=1}^N \varepsilon_k\otimes a_k
     \right\|_{L_p(\Omega;S_p^m)}.
\end{equation}
If $L=0$, then
$\sum_{\ell=1}^N\varepsilon_\ell(\omega)a_\ell=0$ for every
$\omega\in\Omega$, since $\mathbb P(\{\omega\})=2^{-N}>0$.
By $\mathbb E(\varepsilon_k\varepsilon_\ell)=\delta_{k\ell}$, we have
\[
 a_k
 =\sum_{\ell=1}^N\mathbb E(\varepsilon_k\varepsilon_\ell)a_\ell
 =\mathbb E\Bigl(\varepsilon_k\sum_{\ell=1}^N\varepsilon_\ell a_\ell\Bigr)
 =0\qquad(1\le k\le N).
\]
Thus the estimate in Theorem~\ref{thm:general-all-p-upper}
is immediate. Suppose henceforth that $L>0$.

\medskip
\noindent\textbf{Step 3.} We construct bounded, continuous
functions $b_k:\mathbb S\to\mathbb M_m$, analytic on $\mathbb S^\circ$, such that
\[
 \begin{gathered}
 b_k(\theta)=a_k/L\qquad(1\le k\le N),\\[3pt]
 \rho_{r_j}\bigl(b_1(j+it),\ldots,b_N(j+it)\bigr)\le1
 \qquad(j=0,1,\ t\in\mathbb R).
 \end{gathered}
\]

For each $\omega\in\Omega$, define $A(\omega)$ and take its polar decomposition:
\begin{equation}\label{eq:normalized-rademacher-function}
 \begin{aligned}
 A(\omega)&\coloneqq
 L^{-1}\sum_{k=1}^N\varepsilon_k(\omega)a_k
   \overset{\text{\cite[Theorem~4.10]{Teschl}}}{=}
   V(\omega)|A(\omega)|,\\
 |A(\omega)|&\coloneqq
 (A(\omega)^*A(\omega))^{1/2}.
 \end{aligned}
\end{equation}
Here $V(\omega)$ is a partial isometry in $\mathbb M_m$.
For $z\in\mathbb S$, set
\begin{equation}\label{eq:analytic-coefficient-functions}
 \begin{aligned}
 \mathcal A(z)(\omega)&\coloneqq
 V(\omega)|A(\omega)|^{p(1/2-z\delta)}
 \quad(\omega\in\Omega),\\
 b_k(z)&\coloneqq\mathbb E(\varepsilon_k\mathcal A(z))
 \quad(1\le k\le N,\ z\in\mathbb S).
 \end{aligned}
\end{equation}
Since $\Omega$ is finite, the argument for complex matrix powers in
\cite[Definition~7.16 and proof of Proposition~7.4, p.~60]{TroppMatrix}
applied to \eqref{eq:analytic-coefficient-functions} shows that
$\mathcal A$ and each $b_k$ are bounded and continuous on $\mathbb S$,
and analytic on $\mathbb S^\circ$.
Since
$p(1/2-\theta\delta)=1$ and
$\mathbb E(\varepsilon_k\varepsilon_\ell)=\delta_{k\ell}$,
it follows that
\begin{equation}\label{eq:coefficient-interpolation-value}
 b_k(\theta)
 =\mathbb E(\varepsilon_k A)
 =L^{-1}\sum_{\ell=1}^N
   \mathbb E(\varepsilon_k\varepsilon_\ell)a_\ell
 =a_k/L.
\end{equation}
For each $\xi\in\mathbb C^m$, apply scalar Bessel's
inequality \cite[Eq.~(1.11), p.~45]{Teschl}
to each coordinate of $\mathcal A(z)\xi$ and
$\mathcal A(z)^*\xi$, using the orthonormal family
$(\varepsilon_k)_{k=1}^N$ in $L_2(\Omega)$. Summing over the
coordinates gives the following inequalities:
\begin{equation}\label{eq:coefficient-bessel}
 \sum_{k=1}^N b_k(z)^*b_k(z)
 \le\mathbb E(\mathcal A(z)^*\mathcal A(z)),\qquad
 \sum_{k=1}^N b_k(z)b_k(z)^*
 \le\mathbb E(\mathcal A(z)\mathcal A(z)^*).
\end{equation}
The identity $V(\omega)^*V(\omega)=s(|A(\omega)|)$ gives
$|\mathcal A(j+it)(\omega)|=|A(\omega)|^{p/r_j}$,
where $r_0=2$ and $r_1=r$ as defined in
Lemma~\ref{lem:haagerup-analytic-boundary-compactness}.
Hence, for $j=0,1$ and $t\in\mathbb R$,
\begin{equation}\label{eq:coefficient-boundary-normalization}
 \left\|\mathcal A(j+it)\right\|_{L_{r_j}(\Omega;S_{r_j}^m)}^{r_j}
 =\mathbb E\operatorname{tr}|A|^p
 \overset{\eqref{eq:normalized-rademacher-function}}{=}
 L^{-p}\mathbb E\operatorname{tr}
 \left|\sum_{k=1}^N\varepsilon_k a_k\right|^p
 \overset{\eqref{eq:coefficient-normalizing-factor}}{=}1.
\end{equation}
Fix $j\in\{0,1\}$ and put $z=j+it$.
Since $r_j\ge2$, it follows from \eqref{eq:coefficient-bessel},
the triangle inequality, and scalar H\"older's inequality that
\begin{equation}\label{eq:coefficient-square-estimate}
 \begin{aligned}
 &\max\left\{
 \left\|\sum_k b_k(z)^*b_k(z)\right\|_{S_{r_j/2}^m},
 \left\|\sum_k b_k(z)b_k(z)^*\right\|_{S_{r_j/2}^m}\right\}\\
 &\quad\le\mathbb E\left\|\mathcal A(z)\right\|_{S_{r_j}^m}^{2}
 \le\bigl(\mathbb E\left\|\mathcal A(z)\right\|_{S_{r_j}^m}^{r_j}\bigr)^{2/r_j}
 \overset{\eqref{eq:coefficient-boundary-normalization}}{=}1.
 \end{aligned}
\end{equation}
It follows from \eqref{eq:coefficient-square-estimate}
and the definition \eqref{eq:rho-definition} that
\begin{equation}\label{eq:general-analytic-coefficients}
 \rho_{r_j}(b(j+it))
 \overset{\eqref{eq:coefficient-square-estimate}}{\le}
 1
 \overset{\eqref{eq:coefficient-boundary-normalization}}{=}
 \left\|\mathcal A(j+it)\right\|_{L_{r_j}(\Omega;S_{r_j}^m)}
 \quad(j=0,1,\ t\in\mathbb R),
\end{equation}
where $b=(b_1,\ldots,b_N)$.

\medskip
\noindent\textbf{Step 4.} We define $D,H,H_0,H_1$ by
\begin{equation}\label{eq:sum-function-definitions}
 \begin{gathered}
 D\coloneqq\mathbf{1}_{\mathbb M_m}\otimes d,\qquad
 (\operatorname{tr}\otimes\operatorname{Tr})(D)
 =\operatorname{tr}(\mathbf{1}_{\mathbb M_m})\operatorname{Tr}d=m,\\
 H(z)\coloneqq\sum_{k=1}^N b_k(z)\otimes K_k(z),\qquad
 H_j(t)\coloneqq\sum_{k=1}^N b_k(j+it)\otimes v_{k,j}(t),
 \end{gathered}
\end{equation}
and prove that
\[
 \int_{\mathbb R}\left\|H_0(t)\right\|_2^2\,d\nu_0(t)\le5.
\]

By \eqref{eq:theorem-K-boundary-values}, we have
\begin{equation}\label{eq:sum-boundary-values}
 H(j+it)
 =\sum_{k=1}^N b_k(j+it)\otimes d^{j\delta}v_{k,j}(t)
 =D^{j\delta}H_j(t)\quad(j=0,1).
\end{equation}
Since
$\operatorname{tr}(B^*B)=\operatorname{tr}(BB^*)$ for every
$B\in\mathbb M_m$, it follows from \eqref{eq:rho-definition}
and \eqref{eq:general-analytic-coefficients} with $j=0$ that
\[
 \begin{aligned}
 \sum_{k=1}^N\left\|b_k(it)\right\|_{S_2^m}^2
 &=\operatorname{tr}\left(\sum_{k=1}^N b_k(it)^*b_k(it)\right)\\
 &=\rho_2(b(it))^2
 \overset{\eqref{eq:general-analytic-coefficients}}{\le}1,
 \qquad t\in\mathbb R.
 \end{aligned}
\]
In particular,
\[
 \left\|b_k(it)\right\|_{S_2^m}
 \le\left(\sum_{\ell=1}^N\left\|b_\ell(it)\right\|_{S_2^m}^2\right)^{1/2}
 \le1\quad(1\le k\le N,\ t\in\mathbb R).
\]
Since $\left\|v_{k,0}(t)\right\|_2\le2$ by
\eqref{eq:theorem-boundary-limits} and $\nu_0(\mathbb R)=1$,
it follows that
\[
 \begin{aligned}
 &\sum_{k=1}^N\int_{\mathbb R}
   \left\|b_k(it)\right\|_2^2\left\|v_{k,0}(t)\right\|_2^2\,d\nu_0(t)\le4,\\
 &\bigl|\operatorname{tr}(b_i(it)^*b_k(it))\bigr|
   \le\left\|b_i(it)\right\|_2\left\|b_k(it)\right\|_2\le1\quad(i\ne k).
 \end{aligned}
\]
For $q=2$ and $j_2=0$, \eqref{eq:Phi-q-definition} reads
\[
 \Phi_2(i,k)
 =\int_{\mathbb R}\left|
 \operatorname{Tr}\bigl(v_{i,0}(t)^*v_{k,0}(t)\bigr)
 \right|\,d\nu_0(t)\qquad(i\ne k).
\]
By the definition of $H_0(t)$, we have
\begin{equation}\label{eq:lower-boundary-estimate}
 \begin{aligned}
 \int_{\mathbb R}\left\|H_0(t)\right\|_2^2\,d\nu_0(t)
 &=\sum_{i,k=1}^N\int_{\mathbb R}
   \operatorname{tr}\big(b_i(it)^*b_k(it)\big)
   \operatorname{Tr}\big(v_{i,0}(t)^*v_{k,0}(t)\big)\,d\nu_0(t)\\
 &\le 4+\sum_{\substack{1\le i,k\le N\\i\ne k}}\Phi_2(i,k)
 \overset{\eqref{eq:general-analytic-moment-selection}}{\le}5.
 \end{aligned}
\end{equation}

\medskip
\noindent\textbf{Step 5.} With
$A_r\coloneqq8\sqrt2\,r\kappa_r+2^{1/r}$, we prove that
\[
 \left(\int_{\mathbb R}\left\|H_1(t)\right\|_r^r\,d\nu_1(t)\right)^{1/r}
 \le A_r.
\]

The boundary line $\{1+it:t\in\mathbb R\}$ of
$\mathbb S$ corresponds to $j=1$, for which
$r_j=r_1=r\overset{\eqref{eq:theorem-interpolation-parameters}}{=}
2\lfloor p/2\rfloor+2>p$.
To estimate $H_1$, consider
$\mathcal Q\coloneqq L_\infty(\mathbb R,\nu_1)\overline\otimes
\mathbb{M}_m\overline\otimes\mathcal M$ and set
\[
 \begin{gathered}
 z_k(t)\coloneqq b_k(1+it)\otimes v_{k,1}(t)
 \quad(1\le k\le N),\\
 H_1(t)=\sum_{k=1}^N z_k(t)\qquad(t\in\mathbb R).
 \end{gathered}
\]
We use the identification
\cite[immediately before Theorem~6.2]{HJX}
\begin{equation}\label{eq:upper-bochner-norm}
 \begin{gathered}
 L_r(\mathcal Q)
 =L_r(\nu_1;L_r(\mathbb{M}_m\overline\otimes\mathcal M)),\\
 \left\|f\right\|_{L_r(\mathcal Q)}
 =\left(\int_{\mathbb R}\left\|f(t)\right\|_r^r\,d\nu_1(t)\right)^{1/r}
 \quad(f\in L_r(\mathcal Q)).
 \end{gathered}
\end{equation}
Fix $t\in\mathbb R$ and set
$B_k\coloneqq b_k(1+it)$ and
$V_k\coloneqq v_{k,1}(t)$. Then
\[
 \begin{aligned}
 z_k(t)^*z_k(t)&=B_k^*B_k\otimes V_k^*V_k,\\
 z_k(t)z_k(t)^*&=B_kB_k^*\otimes V_kV_k^*.
 \end{aligned}
\]
Taking $j=1$ in \eqref{eq:general-analytic-coefficients}
and \eqref{eq:theorem-boundary-limits}, we obtain
\[
 \rho_r(B_1,\ldots,B_N)\le1,\qquad
 \max_{1\le k\le N}\left\|V_k\right\|_r\le2.
\]
Since $B_k^*B_k,B_kB_k^*\in(\mathbb M_m)_+$ and
$V_k^*V_k,V_kV_k^*\in L_{r/2}(\mathcal M)_+$,
it follows from Lemma~\ref{lem:positive-tensor-high-p} with
$\mathcal N=\mathcal M$, $s=r/2$, and
\[
 (a_k,b_k)=(B_k^*B_k,V_k^*V_k)
 \quad\text{or}\quad
 (a_k,b_k)=(B_kB_k^*,V_kV_k^*)
\]
that
\begin{equation}\label{eq:upper-pointwise-column-square}
 \begin{aligned}
 \left\|\sum_k B_k^*B_k\otimes V_k^*V_k\right\|_{r/2}
 &\le\left\|\sum_kB_k^*B_k\right\|_{S_{r/2}^m}
       \max_k\left\|V_k^*V_k\right\|_{r/2}\\
 &=\left\|\left(\sum_kB_k^*B_k\right)^{1/2}\right\|_{S_r^m}^{2}
       \max_k\left\|V_k\right\|_r^2
 \overset{\eqref{eq:rho-definition}}{\le}
       4\rho_r(B_1,\ldots,B_N)^2
       \overset{\eqref{eq:general-analytic-coefficients}}{\le}4,
 \end{aligned}
\end{equation}
\begin{equation}\label{eq:upper-pointwise-row-square}
 \begin{aligned}
 \left\|\sum_k B_kB_k^*\otimes V_kV_k^*\right\|_{r/2}
 &\le\left\|\sum_kB_kB_k^*\right\|_{S_{r/2}^m}
       \max_k\left\|V_kV_k^*\right\|_{r/2}\\
 &=\left\|\left(\sum_kB_kB_k^*\right)^{1/2}\right\|_{S_r^m}^{2}
       \max_k\left\|V_k\right\|_r^2
 \overset{\eqref{eq:rho-definition}}{\le}
       4\rho_r(B_1,\ldots,B_N)^2
       \overset{\eqref{eq:general-analytic-coefficients}}{\le}4.
 \end{aligned}
\end{equation}
Since $\nu_1(\mathbb R)=1$, integrating the $r/2$ powers
in \eqref{eq:upper-pointwise-column-square} and
\eqref{eq:upper-pointwise-row-square} yields
\begin{equation}\label{eq:upper-integrated-squares}
 \begin{aligned}
 \left\|\sum_k z_k^*z_k\right\|_{L_{r/2}(\mathcal Q)}
 &=\left(\int_{\mathbb R}
  \left\|\sum_k B_k^*B_k\otimes V_k^*V_k\right\|_{r/2}^{r/2}
  \,d\nu_1(t)\right)^{2/r}\le4,\\
 \left\|\sum_k z_kz_k^*\right\|_{L_{r/2}(\mathcal Q)}
 &=\left(\int_{\mathbb R}
  \left\|\sum_k B_kB_k^*\otimes V_kV_k^*\right\|_{r/2}^{r/2}
  \,d\nu_1(t)\right)^{2/r}\le4.
 \end{aligned}
\end{equation}
By the triangle inequality in $L_{r/2}(\mathcal Q)$,
it follows that
\begin{equation}\label{eq:upper-square-sum}
 \left\|\left(\sum_k(z_k^*z_k+z_kz_k^*)\right)^{1/2}
       \right\|_{L_r(\mathcal Q)}^2\le4+4=8.
\end{equation}
On the other hand, since
$0\le B_i^*B_i\le\sum_{k=1}^N B_k^*B_k$ for $1\le i\le N$,
monotonicity of the Schatten $r/2$-norm on positive matrices gives
\[
 \begin{aligned}
 \left\|b_i(1+it)\right\|_{S_r^m}^{2}
 &=\left\|B_i^*B_i\right\|_{S_{r/2}^m}
 \le\left\|\sum_{k=1}^N B_k^*B_k\right\|_{S_{r/2}^m}\\
 &=\left\|\left(\sum_{k=1}^N B_k^*B_k\right)^{1/2}\right\|_{S_r^m}^{2}
 \overset{\eqref{eq:rho-definition}}{\le}\rho_r(B_1,\ldots,B_N)^2
 \overset{\eqref{eq:general-analytic-coefficients}}{\le}1.
 \end{aligned}
\]
Thus $\left\|b_i(1+it)\right\|_{S_r^m}\le1$ for every
$1\le i\le N$ and $t\in\mathbb R$.
For pairwise distinct $i_1,\ldots,i_r\in\{1,\ldots,N\}$,
H\"older's inequality gives
\[
 \left|\operatorname{tr}\big(
 b_{i_1}(1+it)^*b_{i_2}(1+it)\cdots
 b_{i_{r-1}}(1+it)^*b_{i_r}(1+it)\big)\right|
 \le\prod_{\ell=1}^r\left\|b_{i_\ell}(1+it)\right\|_{S_r^m}\le1.
\]
For pairwise distinct $i_1,\ldots,i_r$, the trace
of $z_{i_1}^*z_{i_2}\cdots z_{i_{r-1}}^*z_{i_r}$ is
\begin{equation}\label{eq:upper-trace-factorization}
 \begin{aligned}
 &\operatorname{Tr}_{\mathcal Q}
      (z_{i_1}^*z_{i_2}\cdots z_{i_{r-1}}^*z_{i_r})\\
 &\quad=\int_{\mathbb R}
   \operatorname{tr}\big(b_{i_1}(1+it)^*b_{i_2}(1+it)
                       \cdots b_{i_{r-1}}(1+it)^*b_{i_r}(1+it)\big)\\
 &\qquad\qquad\cdot
   \operatorname{Tr}\big(v_{i_1,1}(t)^*v_{i_2,1}(t)
                       \cdots v_{i_{r-1},1}(t)^*v_{i_r,1}(t)\big)\,d\nu_1(t).
 \end{aligned}
\end{equation}
Recalling the definition of $\Phi_r$
(see \eqref{eq:Phi-q-definition}), we have
\begin{equation}\label{eq:upper-distinct-sum}
 \Delta_r(z_1,\ldots,z_N)
 \overset{\eqref{eq:upper-trace-factorization}}{\le}
 \sum_{\mathbf i\in\mathcal I_r}\Phi_r(\mathbf i)
 \overset{\eqref{eq:general-analytic-moment-selection}}{\le}1.
\end{equation}
Since $r\in2\mathbb N$ and
$r\overset{\eqref{eq:theorem-interpolation-parameters}}{\ge}4$, we take
$\mathcal N=\mathcal Q$, $p=r$, and $u_k=z_k$
in Lemma~\ref{lem:even-haagerup-upper} to obtain
\begin{equation}\label{eq:upper-boundary-estimate}
 \begin{aligned}
 &\qquad\left(\int_{\mathbb R}\left\|H_1(t)\right\|_r^r\,d\nu_1(t)\right)^{1/r}
 \mathrel{\overset{\eqref{eq:sum-function-definitions},\,\eqref{eq:upper-bochner-norm}}{=}}
 \left\|\sum_k z_k\right\|_{L_r(\mathcal Q)}
 \\
 &\mathrel{\overset{\mathclap{\text{Lemma~\ref{lem:even-haagerup-upper}}}}{\le}}
 4r\kappa_r
 \left\|\left(\sum_{k=1}^N(z_k^*z_k+z_kz_k^*)\right)^{1/2}
 \right\|_{L_r(\mathcal Q)}
 +2^{1/r}\Delta_r(z_1,\ldots,z_N)^{1/r}\\
 &\mathrel{\overset{\mathclap{\eqref{eq:upper-square-sum}}}{\le}}
 4r\kappa_r\sqrt8+2^{1/r}\Delta_r(z_1,\ldots,z_N)^{1/r}\\
 &\mathrel{\overset{\mathclap{\eqref{eq:upper-distinct-sum}}}{\le}}
 4r\kappa_r\sqrt8+2^{1/r}=A_r.
 \end{aligned}
\end{equation}

\medskip
\noindent\textbf{Step 6.} With
$a\coloneqq(a_1,\ldots,a_N)$ and
$c_p\coloneqq5^{(1-\theta)/2}A_r^\theta\kappa_p$,
we apply Lemma~\ref{lem:weighted-boundary-interpolation} to $H$ to obtain
\[
 \left\|\sum_{k=1}^N a_k\otimes x_k\right\|_p
 \le c_p\rho_p(a).
\]

We have
\[
 H(\theta)
 \overset{\eqref{eq:sum-function-definitions}}{=}
 \sum_{k=1}^N b_k(\theta)\otimes K_k(\theta)
 \overset{\eqref{eq:coefficient-interpolation-value},\,\eqref{eq:theorem-K-boundary-values}}{=}
 D^{\theta\delta}L^{-1}\sum_{k=1}^N a_k\otimes x_k.
\]
For $1\le q<\infty$, the tensor product map
\begin{equation}\label{eq:elementary-tensor-norm}
 \begin{gathered}
 S_q^m\times L_q(\mathcal M)\longrightarrow
 L_q(\mathbb M_m\overline\otimes\mathcal M),
 \qquad (a,x)\longmapsto a\otimes x,\\
 \left\|a\otimes x\right\|_q=\left\|a\right\|_{S_q^m}\left\|x\right\|_q,
 \end{gathered}
\end{equation}
is bounded and bilinear. Each $b_k$ is bounded and continuous on $\mathbb S$,
and analytic on $\mathbb S^\circ$
(see \cite[proof of Proposition~7.4, p.~60]{TroppMatrix}).
It follows from Lemma~\ref{lem:haagerup-analytic-boundary-compactness}
and \cite[Proposition~1.2.3]{MennickenMoller} that
$H$ in \eqref{eq:sum-function-definitions} is bounded on $\mathbb S$
and analytic on $\mathbb S^\circ$.
Since $b_k$ is norm-continuous and $K_k$ is bounded and weakly continuous
on $\mathbb S$, $H$ is weakly continuous on $\mathbb S$.
\par
Moreover, Lemma~\ref{lem:haagerup-analytic-boundary-compactness}
gives bounded, strongly measurable $v_{k,j}$.
The continuity of the tensor product map and of $t\mapsto b_k(j+it)$
therefore implies that $H_j$ in \eqref{eq:sum-function-definitions}
is bounded and strongly measurable.
The density $D$ is faithful and satisfies
\[
 (\operatorname{tr}\otimes\operatorname{Tr})(D)
 \overset{\eqref{eq:sum-function-definitions}}{=}m<\infty.
\]
Weak continuity of $H$ gives, in
$L_2(\mathbb M_m\overline\otimes\mathcal M)$,
\[
 H(s+it)\xrightarrow[\substack{s\to j\\0<s<1}]{w}H(j+it)
 \overset{\eqref{eq:sum-boundary-values}}{=}D^{j\delta}H_j(t)
 \quad(j=0,1,\ t\in\mathbb R).
\]
We apply Lemma~\ref{lem:weighted-boundary-interpolation} with
$\mathcal N=\mathbb M_m\overline\otimes\mathcal M$,
$d=D$, and
\[
 y\coloneqq L^{-1}\sum_{k=1}^N a_k\otimes x_k,
 \qquad H(\theta)=D^{\theta\delta}y.
\]
It follows that
\[
 \begin{aligned}
 &\qquad
 \left\|\sum_{k=1}^N a_k\otimes x_k\right\|_p
 =L\left\|y\right\|_p\\
 &\overset{\mathclap{\eqref{eq:integrated-kosaki}}}{\le}
 L\left(\int_{\mathbb R}\left\|H_0(t)\right\|_2^2\,d\nu_0(t)\right)^{(1-\theta)/2}
 \left(\int_{\mathbb R}\left\|H_1(t)\right\|_r^r\,d\nu_1(t)\right)^{\theta/r}\\
 &\overset{\mathclap{\eqref{eq:lower-boundary-estimate},\,\eqref{eq:upper-boundary-estimate}}}{\le}
 5^{(1-\theta)/2}A_r^\theta L\\
 &\overset{\mathclap{\eqref{eq:coefficient-normalizing-factor},\,\eqref{eq:khintchine}}}{\le}
 c_p\rho_p(a).
 \end{aligned}
\]

Here $L\overset{\eqref{eq:khintchine}}{\le}\kappa_p\rho_p(a)$,
and $r,\theta$ depend only on $p$ by
\eqref{eq:theorem-interpolation-parameters}.
\end{proof}

\subsection{Completely complemented subspaces}\label{sec:models}

\leavevmode We denote
$\langle\xi,\eta\rangle_2\coloneqq\operatorname{Tr}(\xi^*\eta)$ for
$\xi,\eta\in L_2(\mathcal M)$.

The next lemma is a direct consequence of the factorization result
of Junge, Nielsen and Oikhberg \cite[Lemma~2.9]{JNO}.
We shall use it in Lemma~\ref{lem:spectral-analysis} to obtain
completely bounded maps into $C_p$ and $R_p$.

\begin{lemma}\label{lem:weighted-analysis}
Let $2<p<\infty$, put $r\coloneqq2p/(p-2)$, and let
$h\in L_r(\mathcal M)_+$. Suppose $(u_n)\subset L_2(\mathcal M)$
and $0\le B<\infty$ satisfy
\[
 \sum_{n=1}^\infty|\langle u_n,\xi\rangle_2|^2
 \le B^2\left\|\xi\right\|_2^2\qquad(\xi\in L_2(\mathcal M)).
\]
Then the maps
\[
 \begin{aligned}
 S_R:L_p(\mathcal M)&\longrightarrow C_p,
 &S_Rx&\coloneqq(\langle u_n,xh\rangle_2)_{n\ge1},\\
 S_L:L_p(\mathcal M)&\longrightarrow R_p,
 &S_Lx&\coloneqq(\langle u_n,hx\rangle_2)_{n\ge1}
 \end{aligned}
\]
satisfy $\max\{\left\|S_R\right\|_{\mathrm{cb}},\left\|S_L\right\|_{\mathrm{cb}}\}
\le B\left\|h\right\|_r$.
\end{lemma}

\begin{proof}
The operator
\[
 W:L_2(\mathcal M)\longrightarrow\ell_2,\qquad
 W\xi\coloneqq(\langle u_n,\xi\rangle_2)_{n\ge1},
\]
satisfies $\left\|W\right\|\le B$.
It follows from \cite[Lemma~2.9]{JNO}, applied to
$S_Rx=W(xh)$ and $S_Lx=W(hx)$, that
\begin{equation}\label{eq:weighted-analysis-cb}
 \max\{\left\|S_R\right\|_{\mathrm{cb}},\left\|S_L\right\|_{\mathrm{cb}}\}
 \le\left\|W\right\|\,\left\|h\right\|_r\le B\left\|h\right\|_r.
\end{equation}
\end{proof}

For $2<p<\infty$ and $0\le\lambda\le1$, set
\[
 E_\lambda\coloneqq
 \begin{cases}
 R_p,&\lambda=0,\\
 C_p\cap R_p,&0<\lambda<1,\\
 C_p,&\lambda=1.
 \end{cases}
\]

The orthonormal sequence and completely bounded map
constructed below will be used to obtain the projection $P$ in the proof
of Theorem~\ref{thm:general-high-p-relative-complement}.

\begin{lemma}\label{lem:spectral-analysis}
Let $2<p<\infty$, put $r\coloneqq2p/(p-2)$, and let
$F\subset L_p(\mathcal M)$ be an infinite-dimensional closed
subspace. Suppose that, for some $h\in L_r(\mathcal M)_+$,
\[
 \left\|x\right\|_H\coloneqq\big(\left\|xh\right\|_2^2+\left\|hx\right\|_2^2\big)^{1/2}
\]
is a Hilbert norm satisfying
$c\left\|x\right\|_p\le\left\|x\right\|_H\le C\left\|x\right\|_p$ for all $x\in F$, where
$0<c\le C<\infty$. Let $A$ be the positive
contraction on $H\coloneqq(F,\left\|\cdot\right\|_H)$ defined by
\[
 \langle Ax,y\rangle_H=\langle xh,yh\rangle_2.
\]
Let $K\subset H$ be an infinite-dimensional closed
reducing subspace for $A$,\footnote{A closed subspace $K\subset H$
reduces $A$ if $AK\subset K$ and
$AK^{\perp_H}\subset K^{\perp_H}$; equivalently,
$AP_K=P_KA$, where $P_K$ is the orthogonal projection onto $K$.} and let
$\lambda\in\sigma_{\mathrm{ess}}(A|_K)\subset[0,1]$.\footnote{For a bounded
self-adjoint operator $T$ on a Hilbert space $\mathcal H$,
$\sigma_{\mathrm{ess}}(T)$ denotes its essential spectrum, obtained from
$\sigma(T)$ by removing all isolated eigenvalues of finite multiplicity.
Equivalently,
\[
 \sigma_{\mathrm{ess}}(T)
 =\{\mu\in\mathbb R:
 \dim(\mathbf1_{(\mu-\varepsilon,\mu+\varepsilon)}(T)\mathcal H)=\infty
 \text{ for every }\varepsilon>0\};
\]
see \cite[Section~6.4, Eq.~(6.30)]{Teschl}.}
There is a sequence $(y_n)\subset K$ such that
\begin{equation}\label{eq:spectral-selection}
 \langle y_i,y_j\rangle_H=\delta_{ij}\quad(i,j\in\mathbb N),
 \qquad
 \sum_{n=1}^\infty\left\|(A-\lambda I_H)y_n\right\|_H<\infty.
\end{equation}
Moreover, the map
\[
 \Phi:F\longrightarrow E_\lambda,\qquad
 \Phi x\coloneqq(\langle y_n,x\rangle_H)_{n\ge1}
\]
is completely bounded with respect to the matrix norms $S_p^m[F]$,
$m\ge1$, defined in the Introduction.
\end{lemma}

\begin{proof}
Put
\[
 I\coloneqq\left[\frac{\lambda}{2},\frac{1+\lambda}{2}\right],
 \qquad Q\coloneqq\mathbf1_I(A).
\]
Choose $0<\varepsilon_n\le2^{-n}$ such that
$[\lambda-\varepsilon_n,\lambda+\varepsilon_n]\cap[0,1]\subset I$.
Since $\lambda\in\sigma_{\mathrm{ess}}(A|_K)$, each spectral projection
\[
 Q_n\coloneqq\mathbf1_{[\lambda-\varepsilon_n,\lambda+\varepsilon_n]}(A|_K)
\]
has infinite-dimensional range.
We can therefore choose unit vectors successively so that
\[
 y_n\in Q_nK\cap\operatorname{span}\{y_1,\ldots,y_{n-1}\}^{\perp_H},
 \qquad Qy_n=y_n,
 \qquad\left\|(A-\lambda I_H)y_n\right\|_H\le\varepsilon_n.
\]
Hence \eqref{eq:spectral-selection} holds. Define
\[
 J_R:H\longrightarrow L_2(\mathcal M),\quad J_Rx\coloneqq xh,
 \qquad
 J_L:H\longrightarrow L_2(\mathcal M),\quad J_Lx\coloneqq hx.
\]
Since
$\left\|J_Rx\right\|_2^2+\left\|J_Lx\right\|_2^2
=\left\|x\right\|_H^2$ for every $x\in H$, the maps $J_R$ and $J_L$
are bounded. Recalling that
$\left\|u\right\|_H^2=\left\|uh\right\|_2^2+\left\|hu\right\|_2^2$
for $u\in H$, the complex polarization identity yields, for $x,y\in H$,
\[
 \begin{aligned}
 \langle x,y\rangle_H
 &=\frac14\sum_{k=0}^3(-\mathrm i)^k
   \left\|x+\mathrm i^k y\right\|_H^2\\
 &=\frac14\sum_{k=0}^3(-\mathrm i)^k
   \left(\left\|xh+\mathrm i^k yh\right\|_2^2
        +\left\|hx+\mathrm i^k hy\right\|_2^2\right)\\
 &=\langle xh,yh\rangle_2+\langle hx,hy\rangle_2.
 \end{aligned}
\]
For $J\in\{J_R,J_L\}$, the adjoint $J^*:L_2(\mathcal M)\to H$
is defined by\footnote{For fixed $\xi\in L_2(\mathcal M)$,
the linear functional
\[
 y\longmapsto\langle\xi,Jy\rangle_2\qquad(y\in H)
\]
is bounded because
\[
 |\langle\xi,Jy\rangle_2|
 \le\left\|\xi\right\|_2\left\|Jy\right\|_2
 \le\left\|\xi\right\|_2\left\|y\right\|_H.
\]
By the Riesz representation theorem, there is a unique $z\in H$ such that
\[
 \langle z,y\rangle_H=\langle\xi,Jy\rangle_2\qquad(y\in H).
\]
Defining $J^*\xi\coloneqq z$, we obtain
\[
 \langle J^*\xi,y\rangle_H=\langle\xi,Jy\rangle_2.
\]
}
\[
 \langle J^*\xi,y\rangle_H=\langle\xi,Jy\rangle_2
 \qquad(\xi\in L_2(\mathcal M),\ y\in H).
\]
Taking $(J,\xi)=(J_R,xh)$ and $(J,\xi)=(J_L,hx)$, and using the definition
$\langle Ax,y\rangle_H=\langle xh,yh\rangle_2$ of $A$, we obtain,
for all $x,y\in H$,
\begin{align}
 \langle J_R^*J_Rx,y\rangle_H
 &=\langle J_Rx,J_Ry\rangle_2
 =\langle xh,yh\rangle_2
 =\langle Ax,y\rangle_H,\label{eq:spectral-right-form}\\
 \langle J_L^*J_Lx,y\rangle_H
 &=\langle J_Lx,J_Ly\rangle_2
 =\langle hx,hy\rangle_2\notag\\
 &=\langle x,y\rangle_H-\langle Ax,y\rangle_H
 =\langle (I_H-A)x,y\rangle_H.\label{eq:spectral-left-form}
\end{align}
It follows that
\begin{equation}\label{eq:spectral-multiplication}
 J_R^*J_R=A,\qquad J_L^*J_L=I_H-A.
\end{equation}

If $\lambda>0$, define the bounded Borel function $f:[0,1]\to[0,\infty)$ by
\[
 f(t)\coloneqq
 \begin{cases}t^{-1},&t\in I,\\0,&t\notin I,\end{cases}
 \qquad u_n\coloneqq (f(A)y_n)h.
\]
The definition of $f$ gives
\[
 t f(t)=
 \begin{cases}
 1,&t\in I,\\
 0,&t\in[0,1]\setminus I
 \end{cases}
 =\mathbf1_I(t)\qquad(0\le t\le1).
\]
By Borel functional calculus
\cite[Theorem~3.1]{Teschl},
\[
 Af(A)=(t\mapsto t f(t))(A)=\mathbf1_I(A)=Q.
\]
Since $Qy_n=y_n$, the definitions of $u_n$ and $A$ give, for every $x\in F$,
\begin{equation}\label{eq:spectral-coefficient-identity}
 \begin{aligned}
 \langle u_n,xh\rangle_2
 &=\langle(f(A)y_n)h,xh\rangle_2
 =\langle Af(A)y_n,x\rangle_H\\
 &=\langle Qy_n,x\rangle_H
 =\langle y_n,x\rangle_H.
 \end{aligned}
\end{equation}
Since $f$ is real-valued, $f(A)^*=f(A)$. For $\xi\in L_2(\mathcal M)$, the definition
of $u_n$ and the adjoint identity give
\[
 \begin{aligned}
 \langle u_n,\xi\rangle_2
 &=\langle J_Rf(A)y_n,\xi\rangle_2
 =\langle f(A)y_n,J_R^*\xi\rangle_H\\
 &=\langle y_n,f(A)J_R^*\xi\rangle_H.
 \end{aligned}
\]
Using $\left\|T\right\|^2=\left\|TT^*\right\|$
for $T=f(A)J_R^*:L_2(\mathcal M)\to H$
(see \cite[Section~12.9, Eqs.~(2) and~(7)]{RudinFA}), we obtain
\[
 \begin{aligned}
 \left\|f(A)J_R^*\right\|_{L_2(\mathcal M)\to H}^2
 &=\left\|f(A)J_R^*J_Rf(A)\right\|_{B(H)}
 \overset{\eqref{eq:spectral-multiplication}}{=}
 \left\|f(A)Af(A)\right\|_{B(H)}\\
 &\le\sup_{t\in I}\frac1t=\frac2\lambda,
 \end{aligned}
\]
since $t f(t)^2=f(t)$ for $0\le t\le1$ and
$I=[\lambda/2,(1+\lambda)/2]$.
It follows from Bessel's inequality \cite[Theorem~1.2]{Teschl} that
\begin{equation}\label{eq:spectral-bessel}
 \begin{aligned}
 \sum_n|\langle u_n,\xi\rangle_2|^2
 &=\sum_n|\langle y_n,f(A)J_R^*\xi\rangle_H|^2\\
 &\le\left\|f(A)J_R^*\xi\right\|_H^2
 \le\left\|f(A)J_R^*\right\|_{L_2(\mathcal M)\to H}^2
 \,\left\|\xi\right\|_2^2
 \le\frac{2}{\lambda}\left\|\xi\right\|_2^2.
 \end{aligned}
\end{equation}
It follows from Lemma~\ref{lem:weighted-analysis} and
\eqref{eq:spectral-coefficient-identity} that
\begin{equation}\label{eq:spectral-column-cb}
 \left\|\Phi:F\to C_p\right\|_{\mathrm{cb}}
 \le\sqrt{\frac{2}{\lambda}}\,\left\|h\right\|_r.
\end{equation}

If $\lambda<1$, define $g:[0,1]\to[0,\infty)$ and $(v_n)$ by
\[
 g(t)\coloneqq
 \begin{cases}(1-t)^{-1},&t\in I,\\0,&t\notin I,\end{cases}
 \qquad v_n\coloneqq h(g(A)y_n).
\]
Since $v_n=J_Lg(A)y_n$ and $hx=J_Lx$, the adjoint identity gives
\[
 \begin{aligned}
 \langle v_n,hx\rangle_2
 &=\langle J_Lg(A)y_n,J_Lx\rangle_2
 =\langle J_L^*J_Lg(A)y_n,x\rangle_H\\
 &\overset{\eqref{eq:spectral-multiplication}}{=}
 \langle(I_H-A)g(A)y_n,x\rangle_H
 =\langle Qy_n,x\rangle_H
 =\langle y_n,x\rangle_H
 \qquad(x\in F),
 \end{aligned}
\]
where $(I_H-A)g(A)=\mathbf1_I(A)=Q$ follows from
$(1-t)g(t)=\mathbf1_I(t)$ and Borel functional calculus
\cite[Theorem~3.1]{Teschl}.
For $\xi\in L_2(\mathcal M)$, the identity $g(A)^*=g(A)$ gives
\[
 \langle v_n,\xi\rangle_2
 =\langle J_Lg(A)y_n,\xi\rangle_2
 =\langle g(A)y_n,J_L^*\xi\rangle_H
 =\langle y_n,g(A)J_L^*\xi\rangle_H.
\]
Moreover, for $\xi\in L_2(\mathcal M)$, Bessel's inequality gives
\[
 \sum_n|\langle v_n,\xi\rangle_2|^2
 =\sum_n|\langle y_n,g(A)J_L^*\xi\rangle_H|^2
 \le\left\|g(A)J_L^*\xi\right\|_H^2
 \le\frac{2}{1-\lambda}\left\|\xi\right\|_2^2,
\]
where the last inequality follows from
\[
 \begin{aligned}
 \left\|g(A)J_L^*\right\|_{L_2(\mathcal M)\to H}^2
 &=\left\|g(A)J_L^*J_Lg(A)\right\|_{B(H)}
 \overset{\eqref{eq:spectral-multiplication}}{=}
 \left\|g(A)(I_H-A)g(A)\right\|_{B(H)}\\
 &\le\sup_{t\in I}\frac{1}{1-t}
 =\frac{1}{1-(1+\lambda)/2}
 =\frac{2}{1-\lambda}.
 \end{aligned}
\]

Lemma~\ref{lem:weighted-analysis} therefore gives
\begin{equation}\label{eq:spectral-row-cb}
 \left\|\Phi:F\to R_p\right\|_{\mathrm{cb}}
 \le\sqrt{\frac{2}{1-\lambda}}\,\left\|h\right\|_r.
\end{equation}
For $0<\lambda<1$, it follows from \eqref{eq:intersection-amplification},
\eqref{eq:spectral-column-cb}, and \eqref{eq:spectral-row-cb} that
\[
 \left\|\Phi:F\to C_p\cap R_p\right\|_{\mathrm{cb}}
 \le2\max\{\left\|\Phi:F\to C_p\right\|_{\mathrm{cb}},
           \left\|\Phi:F\to R_p\right\|_{\mathrm{cb}}\}<\infty.
\]
For $\lambda=0$ or $\lambda=1$, use \eqref{eq:spectral-row-cb}
or \eqref{eq:spectral-column-cb}, respectively.
\end{proof}

Using the result of Raynaud and Xu
\cite[Theorem~5.1]{RX}, the next lemma constructs a weight $h$ and
an infinite-dimensional separable closed subspace $H_\infty$ of
an infinite-dimensional Hilbert space $H$, such that $H_\infty$
reduces $A$ and $H_\infty\subset eL_p(\mathcal M)e$,
where $e=s(h)$.
These properties allow us to apply Lemma~\ref{lem:spectral-analysis}.

\begin{lemma}
\label{lem:weighted-reducing-hull}
\emergencystretch=1em
Let $2<p<\infty$ and let $F\subset L_p(\mathcal M)$ be a closed,
infinite-dimensional Hilbertian subspace. There exist $h\in L_r(\mathcal M)_+$,
$r\coloneqq2p/(p-2)$, and an infinite-dimensional separable closed subspace
$H_\infty\subset F$ and constants $0<c\le C<\infty$ such that
\[
 c\left\|x\right\|_p\le\left\|x\right\|_H
 :=\big(\left\|xh\right\|_2^2+\left\|hx\right\|_2^2\big)^{1/2}
 \le C\left\|x\right\|_p\qquad(x\in F).
\]
For $H\coloneqq(F,\left\|\cdot\right\|_H)$ and $A$ defined by
$\langle Ax,y\rangle_H=\langle xh,yh\rangle_2$,
the subspace $H_\infty$ satisfies
\[
 A(H_\infty)\subset H_\infty,\qquad
 A(H_\infty^{\perp_H})\subset H_\infty^{\perp_H}.
\]
Moreover, $e\coloneqq s(h)$ is {$\sigma$-finite} and
$H_\infty\subset eL_p(\mathcal M)e$.
\end{lemma}

\begin{proof}
Since $F$ is Hilbertian and $p>2$, $F$ contains no subspace
isomorphic to $\ell_p$. By
\cite[Theorem~5.1, $(iii)\Leftrightarrow(v)$]{RX}, with $q=2$,
there exist $g\in L_r(\mathcal M)$ and $c_0>0$ such that
\[
 \left\|xg\right\|_2^2+\left\|gx\right\|_2^2
 \ge c_0^2\left\|x\right\|_p^2\qquad(x\in F).
\]
Set $h_0=(g^*g+gg^*)^{1/2}\in L_r(\mathcal M)_+$. By the tracial property of $\operatorname{Tr}$,
\begin{equation}\label{eq:initial-weighted-lower-bound}
 \left\|xh_0\right\|_2^2+\left\|h_0x\right\|_2^2
 =\left\|xg\right\|_2^2+\left\|xg^*\right\|_2^2
   +\left\|gx\right\|_2^2+\left\|g^*x\right\|_2^2
 \ge c_0^2\left\|x\right\|_p^2\qquad(x\in F).
\end{equation}
Equip $F$ with the equivalent Hilbert norm
\begin{equation}\label{eq:initial-weighted-norm}
 \left\|x\right\|_{H_0}^2\coloneqq\left\|xh_0\right\|_2^2+\left\|h_0x\right\|_2^2.
\end{equation}
For $k\in L_r(\mathcal M)_+$, define bounded positive
operators $R_k,L_k$ on $H_0\coloneqq(F,\left\|\cdot\right\|_{H_0})$ by
\begin{equation}\label{eq:weighted-form-operators}
 \langle R_kx,y\rangle_{H_0}=\langle xk,yk\rangle_2,
 \qquad
 \langle L_kx,y\rangle_{H_0}=\langle kx,ky\rangle_2.
\end{equation}
In particular, $R_{h_0}+L_{h_0}=I$. Indeed, for all $x,y\in F$, we have
\[
 \langle(R_{h_0}+L_{h_0})x,y\rangle_{H_0}
 =\langle xh_0,yh_0\rangle_2+\langle h_0x,h_0y\rangle_2
 \overset{\eqref{eq:initial-weighted-norm}}{=}\langle x,y\rangle_{H_0}.
\]

Start with an infinite-dimensional separable closed subspace
$V_1\subset F$. Having constructed $V_n$, choose
a countable dense subset of $V_n$, and let $e_n$ be the join
(supremum) of the left and right support projections of its elements. By
\eqref{eq:sigma-finite-support-reduction}, $e_n$ is $\sigma$-finite
and $V_n\subset L_p(e_n\mathcal M e_n)$. Choose a faithful
$d_n\in L_1(e_n\mathcal M e_n)_+$ with $\operatorname{Tr}d_n=1$ and set
$k_n\coloneqq d_n^{1/r}$. Then $s(k_n)=e_n$ and $\left\|k_n\right\|_r=1$.
Set
$\mathcal T_n\coloneqq
\{R_{h_0},L_{h_0},R_{k_j},L_{k_j}:1\le j\le n\}$ and define
\begin{equation}\label{eq:reducing-subspace-construction}
 \begin{aligned}
 V_{n+1}&\coloneqq
 \overline{\operatorname{span}}
 \{T_1\cdots T_\ell x:\ell\ge0,\ T_i\in\mathcal T_n,\ x\in V_n\},\\
 H_\infty&\coloneqq\overline{\bigcup_{n\ge1}V_n}.
 \end{aligned}
\end{equation}
The closures are taken in $H_0$, and the product $T_1\cdots T_\ell$
is interpreted as the identity when $\ell=0$. There are countably many finite products
of operators in $\mathcal T_n$, so $V_{n+1}$ is separable.
For each fixed $j$, all $V_n$ with $n>j$ are invariant under
$R_{k_j},L_{k_j}$; all $V_n$ with $n\ge2$ are invariant under
$R_{h_0},L_{h_0}$. 
By continuity and self-adjointness on $H_0$, for every
$T\in\bigcup_{n\ge1}\mathcal T_n$,
\begin{equation}\label{eq:initial-reducing-relations}
 T(H_\infty)\subset H_\infty,\qquad
 T(H_\infty^{\perp_{H_0}})\subset H_\infty^{\perp_{H_0}}.
\end{equation}

Since $r>2$, the positive series
\begin{equation}\label{eq:combined-weight}
 h^2\coloneqq h_0^2+\sum_{n\ge1}2^{-n}k_n^2
\end{equation}
converges in $L_{r/2}$ and defines $h\in L_r(\mathcal M)_+$.
Its support satisfies
\[
 e=s(h)=s(h_0)\vee\bigvee_{n\ge1}e_n,\qquad
 V_n\subset eL_p(\mathcal M)e\quad(n\ge1).
\]
The support $s(h_0)$ and every $e_n$ are $\sigma$-finite,
so $e$ is $\sigma$-finite. By \eqref{eq:reducing-subspace-construction},
$H_\infty\subset eL_p(\mathcal M)e$.
For $x\in F$, by \eqref{eq:combined-weight},
\begin{equation}\label{eq:combined-weight-identities}
 \begin{aligned}
 \left\|xh\right\|_2^2&=\left\|xh_0\right\|_2^2
   +\sum_{n\ge1}2^{-n}\left\|xk_n\right\|_2^2,\\
 \left\|hx\right\|_2^2&=\left\|h_0x\right\|_2^2
   +\sum_{n\ge1}2^{-n}\left\|k_nx\right\|_2^2.
 \end{aligned}
\end{equation}
Define operators on $H_0$ by
\begin{equation}\label{eq:weighted-sum-operators}
 R\coloneqq R_{h_0}+\sum_n2^{-n}R_{k_n},\qquad
 L\coloneqq L_{h_0}+\sum_n2^{-n}L_{k_n},
\end{equation}
where the operator-norm convergence follows from H\"older's inequality:
\[
\begin{aligned}
 \langle R_{k_n}x,x\rangle_{H_0}
 &=\left\|xk_n\right\|_2^2
 \overset{\frac12=\frac1p+\frac1r}{\le}
 \left\|k_n\right\|_r^2\left\|x\right\|_p^2
 \overset{\eqref{eq:initial-weighted-lower-bound},\,\eqref{eq:initial-weighted-norm}}{\le}
 c_0^{-2}\left\|k_n\right\|_r^2\left\|x\right\|_{H_0}^2,\\
 \langle L_{k_n}x,x\rangle_{H_0}
 &=\left\|k_nx\right\|_2^2
 \overset{\frac12=\frac1p+\frac1r}{\le}
 \left\|k_n\right\|_r^2\left\|x\right\|_p^2
 \overset{\eqref{eq:initial-weighted-lower-bound},\,\eqref{eq:initial-weighted-norm}}{\le}
 c_0^{-2}\left\|k_n\right\|_r^2\left\|x\right\|_{H_0}^2
 \qquad(x\in F).
\end{aligned}
\]
Since $R_{k_n}$ and $L_{k_n}$ are positive, taking the supremum over
$\left\|x\right\|_{H_0}=1$ gives
\[
 \left\|R_{k_n}\right\|,\ \left\|L_{k_n}\right\|
 \le c_0^{-2}\left\|k_n\right\|_r^2=c_0^{-2}.
\]
Since $H_\infty$ is closed in the Hilbert space $H_0$, we have
$H_0=H_\infty\oplus H_\infty^{\perp_{H_0}}$.
Define a bounded projection $P:H_0\to H_0$ by $P(u+v)=u$ for
$u\in H_\infty$ and $v\in H_\infty^{\perp_{H_0}}$.
Thus
\[
 \begin{aligned}
 &Px\in H_\infty,\qquad x-Px\in H_\infty^{\perp_{H_0}}
 \qquad(x\in F),\\
 &\langle Px,y\rangle_{H_0}=\langle x,Py\rangle_{H_0}
 \qquad(x,y\in F).
 \end{aligned}
\]
By the definition of $\mathcal T_n$,
\[
 \bigcup_{n\ge1}\mathcal T_n
 =\{R_{h_0},L_{h_0},R_{k_n},L_{k_n}:n\ge1\}.
\]
Thus \eqref{eq:initial-reducing-relations} applies to each of these operators
and to their scalar multiples. For every summand $T$ in
\eqref{eq:weighted-sum-operators} and every $x\in F$,
\[
 TPx\in H_\infty,\qquad
 T(I-P)x\in H_\infty^{\perp_{H_0}}.
\]
Since $P$ is the identity on $H_\infty$ and vanishes on
$H_\infty^{\perp_{H_0}}$, we obtain
\[
 PTx=PTPx+PT(I-P)x=TPx+0=TPx\qquad(x\in F).
\]
Passing to the operator-norm limits gives $PR=RP$ and $PL=LP$.
Since $R_{h_0}+L_{h_0}=I$ and $R_{k_n},L_{k_n}\ge0$,
\[
 B\coloneqq R+L
 =I+\sum_{n\ge1}2^{-n}(R_{k_n}+L_{k_n})\ge I.
\]
In particular, $B$ is bounded, positive and invertible, and $PB=BP$.
Define
\begin{equation}\label{eq:renormed-H-inner-product}
 \langle x,y\rangle_H\coloneqq\langle Bx,y\rangle_{H_0}.
\end{equation}
This is an equivalent Hilbert inner product on $F$, since $B$ is bounded
and $B\ge I$. By the definitions of $R$ and $L$, for $x\in F$,
\[
\begin{aligned}
 \left\|x\right\|_H^2
 &\mathrel{
 \overset{\mathclap{\eqref{eq:renormed-H-inner-product}}}{=}}
 \langle Bx,x\rangle_{H_0}
 =\langle Rx,x\rangle_{H_0}+\langle Lx,x\rangle_{H_0}\\
 &\mathrel{
 \overset{\mathclap{\eqref{eq:weighted-form-operators},\,\eqref{eq:weighted-sum-operators}}}{=}}
 \left\|xh_0\right\|_2^2+\left\|h_0x\right\|_2^2
   +\sum_{n\ge1}2^{-n}\bigl(\left\|xk_n\right\|_2^2
                             +\left\|k_nx\right\|_2^2\bigr)\\
 &\mathrel{
 \overset{\mathclap{\eqref{eq:combined-weight-identities}}}{=}}
   \left\|xh\right\|_2^2+\left\|hx\right\|_2^2.
\end{aligned}
\]
Moreover, $B\ge I$ gives
\[
 \left\|x\right\|_{H_0}^2
 =\langle x,x\rangle_{H_0}
 \le\langle Bx,x\rangle_{H_0}=\left\|x\right\|_H^2,
\]
and H\"older's inequality, with $1/2=1/p+1/r$, gives
\[
 \left\|x\right\|_H^2
 =\left\|xh\right\|_2^2+\left\|hx\right\|_2^2
 \le\left\|x\right\|_p^2\left\|h\right\|_r^2
   +\left\|h\right\|_r^2\left\|x\right\|_p^2
 =2\left\|h\right\|_r^2\left\|x\right\|_p^2.
\]
Together with \eqref{eq:initial-weighted-lower-bound}--\eqref{eq:initial-weighted-norm},
these estimates yield
$c_0\left\|x\right\|_p\le\left\|x\right\|_H
\le\sqrt2\left\|h\right\|_r\left\|x\right\|_p$.

Since $PB=BP$, it follows that for all $x,y\in F$,
\[
 \langle Px,y\rangle_H
 =\langle BPx,y\rangle_{H_0}
 =\langle PBx,y\rangle_{H_0}
 =\langle Bx,Py\rangle_{H_0}
 =\langle x,Py\rangle_H.
\]
For $x\in F$ and $z\in H_\infty$, we have $Pz=z$, and hence
\[
 \langle x-Px,z\rangle_H
 =\langle x,z\rangle_H-\langle x,Pz\rangle_H=0.
\]
Thus $Px\in H_\infty$ and $x-Px$ is orthogonal to every vector in
$H_\infty$ with respect to $\langle\cdot,\cdot\rangle_H$.
If $Px=0$, the preceding identity gives
$\langle x,z\rangle_H=0$ for every $z\in H_\infty$;
hence $\ker P\subset H_\infty^{\perp_H}$.
Conversely, if $x\in H_\infty^{\perp_H}$, then $Px=0$.
Thus $\ker P=H_\infty^{\perp_H}$.

For all $x,y\in F$,
\[
\begin{aligned}
 \langle B^{-1}Rx,y\rangle_H
 &\mathrel{
 \overset{\mathclap{\eqref{eq:renormed-H-inner-product}}}{=}}
 \langle BB^{-1}Rx,y\rangle_{H_0}
  =\langle Rx,y\rangle_{H_0}\\
 &\mathrel{
 \overset{\mathclap{\eqref{eq:weighted-sum-operators}}}{=}}
 \langle xh,yh\rangle_2
  =\langle Ax,y\rangle_H.
\end{aligned}
\]
Also, $PB=BP$ implies $PB^{-1}=B^{-1}P$. Consequently,
\begin{equation}\label{eq:reducing-weighted-operator}
 A=B^{-1}R,\qquad
 PA=PB^{-1}R=B^{-1}PR=B^{-1}RP=AP.
\end{equation}
Since $PA=AP$, both $\operatorname{ran}P=H_\infty$ and
$\ker P=H_\infty^{\perp_H}$ are invariant under $A$.
\end{proof}

\subsection{Proof of Theorem~\ref{thm:general-high-p-relative-complement}}
\label{sec:proof-main-theorem}

\begin{proof}
If $F$ is not Hilbertian,
\cite[Corollary~6.3]{RX} gives a subspace
$Z\subset F$ completely isomorphic to $\ell_p$ and a completely
bounded projection $Q:L_p(\mathcal M)\to Z$.
Then $Q|_F:F\to Z$ is the required projection.

Suppose that $F$ is Hilbertian; since $p>2$, $F$ contains no
isomorphic copy of $\ell_p$. Apply
Lemma~\ref{lem:weighted-reducing-hull} and retain its notation
$h$, $H$, $A$, $H_\infty$, and $e=s(h)$.
The restriction $A|_{H_\infty}$ is self-adjoint on the
infinite-dimensional Hilbert space $H_\infty$, so its essential
spectrum is nonempty. Choose
$\lambda\in\sigma_{\mathrm{ess}}(A|_{H_\infty})$.
Apply Lemma~\ref{lem:spectral-analysis} with $K=H_\infty$
to choose $(y_n)\subset H_\infty$ such that
\begin{equation}\label{eq:model-orthonormal-selection}
 \langle y_i,y_j\rangle_H=\delta_{ij},\qquad
 \sum_{n=1}^\infty\left\|(A-\lambda I_H)y_n\right\|_H<\infty,
\end{equation}
and the map is completely bounded:
\[
 \Phi:F\longrightarrow E_\lambda,\qquad
\Phi x\coloneqq(\langle y_n,x\rangle_H)_n.
\]
Here $E_\lambda$ is defined before Lemma~\ref{lem:spectral-analysis}.
Since $(y_n)\subset H_\infty$, the last assertion of
Lemma~\ref{lem:weighted-reducing-hull} gives
$y_n\in eL_p(\mathcal M)e=L_p(e\mathcal M e)$ for every $n$.
Moreover, $s(h)=e$, so $h$ is faithful in $e\mathcal M e$.

\textbf{Case 1:} Suppose first that $0<\lambda<1$.
Orthonormality in \eqref{eq:model-orthonormal-selection}
implies $y_n\rightharpoonup0$ in $H$. Since the inclusion
$H\hookrightarrow L_p(\mathcal M)$ is bounded, $(y_n)$ is bounded
and weakly null in $L_p(\mathcal M)$. Apply
Theorem~\ref{thm:general-all-p-upper} to obtain a subsequence
$(y_{n_k})$ satisfying
\eqref{eq:upper-selection-estimate}. Define
$Q((\alpha_n)_n)\coloneqq(\alpha_{n_k})_k$ and set
\[
 \Phi'\coloneqq Q\Phi:F\longrightarrow C_p\cap R_p,\qquad
 \Phi'x=(\langle y_{n_k},x\rangle_H)_{k\ge1}\quad(x\in F).
\]
For every $m\in\mathbb N$ and every finitely supported family
$(a_n)\subset\mathbb M_m$, we have
\[
 \sum_k a_{n_k}^*a_{n_k}\le\sum_n a_n^*a_n,
 \qquad
 \sum_k a_{n_k}a_{n_k}^*\le\sum_n a_na_n^*.
\]
By \eqref{eq:column-row-amplification} and
\eqref{eq:cb-amplification-criterion}, $Q$ is completely contractive on
each of $C_p$ and $R_p$. By the definition of $C_p\cap R_p$ in
\eqref{eq:intersection-matrix-norm}, $Q$ is also completely contractive
on $C_p\cap R_p$. Thus $\Phi'=Q\Phi$ is completely bounded.
The synthesis map
\[
 T:C_p\cap R_p\longrightarrow F,\qquad
 Te_k\coloneqq y_{n_k}\quad(k\in\mathbb N),
\]
is completely bounded by
Theorem~\ref{thm:general-all-p-upper} and
\eqref{eq:intersection-amplification}.
Indeed, for every
$m,N\in\mathbb N$ and $a\in\mathbb M_m^N$,
\[
\begin{aligned}
 \left\|(\mathrm{id}_{\mathbb M_m}\otimes T)
 \left(\sum_{k=1}^N a_k\otimes e_k\right)\right\|_{S_p^m[F]}
 &=\left\|\sum_{k=1}^N a_k\otimes y_{n_k}\right\|_p\\
 &\overset{\eqref{eq:upper-selection-estimate}}{\le}
 c_p\left(\sup_n\left\|y_n\right\|_p\right)\rho_p(a)\\
 &\overset{\eqref{eq:intersection-amplification}}{\le}
 c_p\left(\sup_n\left\|y_n\right\|_p\right)
 \left\|\sum_{k=1}^N a_k\otimes e_k\right\|_{S_p^m[C_p\cap R_p]}.
\end{aligned}
\]
The orthonormality of
$(y_{n_k})$ in $H$, given by \eqref{eq:model-orthonormal-selection}, yields
$\Phi'Te_k=(\langle y_{n_j},y_{n_k}\rangle_H)_j=e_k$;
hence $\Phi'T=I_{C_p\cap R_p}$, where
$I_{C_p\cap R_p}$ denotes the identity operator on $C_p\cap R_p$. The space
$Z\coloneqq T(C_p\cap R_p)$ is completely isomorphic to $C_p\cap R_p$
with completely bounded inverse $\Phi'|_Z$.
The map $P\coloneqq T\Phi':F\to F$ satisfies
$P^2=P$, $P(F)=Z$ and $\left\|P\right\|_{\mathrm{cb}}<\infty$.

\textbf{Case 2:} For the cases $\lambda=0,1$, set
$\mathcal N\coloneqq e\mathcal M e$.
By the norm equivalence in Lemma~\ref{lem:weighted-reducing-hull},
\[
 c\left\|y_n\right\|_p\le\left\|y_n\right\|_H=1\qquad(n\in\mathbb N).
\]
Thus $(y_n)$ is bounded in $L_p(\mathcal N)$.
Fix a free
ultrafilter $\mathcal U$ on $\mathbb N$.
We use the compatible realizations of $L_q(\mathcal N)_{\mathcal U}$,
$q=p,r,2$, where $r=2p/(p-2)$, as noncommutative $L_q$-spaces over a common von Neumann
algebra $\mathcal A_{\mathcal U}$; see
\cite[Section~1]{RX}. For a bounded sequence
$(u_n)$ in $L_q(\mathcal N)$, write $(u_n)^\bullet$ for its
ultrapower class, so that
\[
 \left\|(u_n)^\bullet\right\|_q
 =\lim_{n\to\mathcal U}\left\|u_n\right\|_q,
 \qquad \widehat x\coloneqq(x,x,\ldots)^\bullet.
\]
The support of the diagonal copy
$\{\widehat{x}:x\in L_p(\mathcal N)\}$
of $L_p(\mathcal N)$ is
\[
 s_{\mathrm{diag}}\coloneqq
 \bigvee_{x\in L_p(\mathcal N)}
 \bigl(s(|\widehat x|)\vee s(|\widehat x^*|)\bigr)
 \in\mathcal A_{\mathcal U}.
\]
Since $h$ is faithful in $\mathcal N$,
\cite[Section~4, immediately before Theorem~4.6]{RX}
gives
\begin{equation}\label{eq:diagonal-support}
 s(\widehat h)=s_{\mathrm{diag}}.
\end{equation}

For $\lambda\in\{0,1\}$, we have
\[
 \begin{aligned}
 \lambda=1:\quad
 \left\|hy_n\right\|_2^2
 &\mathrel{
   \overset{\eqref{eq:spectral-left-form}}{=}}
 \langle(I_H-A)y_n,y_n\rangle_H
 \mathrel{
   \overset{\eqref{eq:model-orthonormal-selection}}{\le}}
 \left\|(I_H-A)y_n\right\|_H\longrightarrow0,\\
 \lambda=0:\quad
 \left\|y_nh\right\|_2^2
 &\mathrel{
   \overset{\eqref{eq:spectral-right-form}}{=}}
 \langle Ay_n,y_n\rangle_H
 \mathrel{
   \overset{\eqref{eq:model-orthonormal-selection}}{\le}}
 \left\|Ay_n\right\|_H\longrightarrow0.
 \end{aligned}
\]
Put $\widetilde y\coloneqq(y_n)^\bullet$.
By \eqref{eq:diagonal-support}, for every $x\in L_p(\mathcal N)$ we have
\[
 \begin{aligned}
 \lambda=1:\quad
 \widehat h\widetilde y=(hy_n)^\bullet=0
 &\ \Longrightarrow\ s_{\mathrm{diag}}\widetilde y=0
 \ \Longrightarrow\ \widehat x^*\widetilde y=0,\\
 \lambda=0:\quad
 \widetilde y\widehat h=(y_nh)^\bullet=0
 &\ \Longrightarrow\ \widetilde y s_{\mathrm{diag}}=0
 \ \Longrightarrow\ \widetilde y\widehat x^*=0.
 \end{aligned}
\]
Thus $\widetilde y$ is left disjoint from the diagonal copy
$\{\widehat{x}:x\in L_p(\mathcal N)\}$ of
$L_p(\mathcal N)$ when $\lambda=1$, and right disjoint when $\lambda=0$.
By \cite[Remarks~2.8(i)]{RX}, after passing to a subsequence there are
$z_n\in L_p(\mathcal N)$ such that $\left\|y_n-z_n\right\|_p\to0$ and the
corresponding one-sided products
$z_i^*z_j$ (when $\lambda=1$) or $z_i z_j^*$ (when $\lambda=0$)
vanish for $i\ne j$. Passing to a further subsequence
and relabeling, we obtain
\begin{equation}\label{eq:endpoint-disjoint-approximation}
 \sum_{n=1}^\infty\left\|y_n-z_n\right\|_p<\infty,
 \qquad
 \begin{cases}
 z_i^*z_j=0,&\lambda=1,\\
 z_i z_j^*=0,&\lambda=0
 \end{cases}
 \quad(i\ne j).
\end{equation}
For a suitable increasing sequence $j_1<j_2<\cdots$, define
$Q((\alpha_n)_n)\coloneqq(\alpha_{j_k})_k$.
Since $E_1=C_p$ and $E_0=R_p$, \eqref{eq:column-row-amplification} and
\eqref{eq:cb-amplification-criterion} give
\[
 \left\|Q:E_\lambda\to E_\lambda\right\|_{\mathrm{cb}}\le1,
 \qquad \left\|Q\Phi\right\|_{\mathrm{cb}}\le\left\|\Phi\right\|_{\mathrm{cb}}<\infty.
\]
We henceforth denote $Q\Phi$ by $\Phi$. For the relabeled sequence $(y_n)$,
this means
\[
 \Phi:F\longrightarrow E_\lambda,\qquad
 \Phi x=(\langle y_n,x\rangle_H)_n\quad(x\in F).
\]

Fix $m,N\in\mathbb N$ and $a_1,\ldots,a_N\in\mathbb M_m$, and put
$X\coloneqq\sum_{n=1}^N a_n\otimes z_n$. Recall that
\begin{equation}\label{eq:endpoint-matrix-norm}
 \left\|\sum_{n=1}^N a_n\otimes e_n\right\|_{S_p^m[E_\lambda]}
 =\begin{cases}
 \displaystyle\left\|\left(\sum_{n=1}^N a_n^*a_n\right)^{1/2}\right\|_{S_p^m},
 &\lambda=1,\\[6pt]
 \displaystyle\left\|\left(\sum_{n=1}^N a_na_n^*\right)^{1/2}\right\|_{S_p^m},
 &\lambda=0.
 \end{cases}
\end{equation}
By \eqref{eq:endpoint-disjoint-approximation},
\[
 \begin{aligned}
 \lambda=1:\quad X^*X&=\sum_{n=1}^N a_n^*a_n\otimes z_n^*z_n,\\
 \lambda=0:\quad XX^*&=\sum_{n=1}^N a_na_n^*\otimes z_nz_n^*.
 \end{aligned}
\]
Applying Lemma~\ref{lem:positive-tensor-high-p} with exponent $p/2$ gives
\begin{equation}\label{eq:endpoint-disjoint-estimate}
 \begin{aligned}
 \left\|X\right\|_p^2
 &\le
 \begin{cases}
 \displaystyle\left\|\sum_{n=1}^N a_n^*a_n\right\|_{S_{p/2}^m}
     \sup_n\left\|z_n^*z_n\right\|_{p/2},&\lambda=1,\\[6pt]
 \displaystyle\left\|\sum_{n=1}^N a_na_n^*\right\|_{S_{p/2}^m}
     \sup_n\left\|z_nz_n^*\right\|_{p/2},&\lambda=0
 \end{cases}\\
 &\overset{\eqref{eq:endpoint-matrix-norm}}{=}
 \left\|\sum_{n=1}^N a_n\otimes e_n\right\|_{S_p^m[E_\lambda]}^2
 \sup_n\left\|z_n\right\|_p^2.
 \end{aligned}
\end{equation}
Since $a_n^*a_n\le\sum_{k=1}^N a_k^*a_k$ and
$a_na_n^*\le\sum_{k=1}^N a_ka_k^*$, \eqref{eq:endpoint-matrix-norm} gives
\[
 \max_{1\le n\le N}\left\|a_n\right\|_{S_p^m}
 \le\left\|\sum_{n=1}^N a_n\otimes e_n\right\|_{S_p^m[E_\lambda]}.
\]
Consequently,
\begin{equation}\label{eq:synthesis-perturbation}
 \begin{aligned}
 \left\|\sum_{n=1}^N a_n\otimes(y_n-z_n)\right\|_p
 &\le\sum_{n=1}^N\left\|a_n\otimes(y_n-z_n)\right\|_p\\
 &\overset{\eqref{eq:elementary-tensor-norm}}{=}
 \sum_{n=1}^N\left\|a_n\right\|_{S_p^m}\left\|y_n-z_n\right\|_p\\
 &\le\left\|\sum_{n=1}^N a_n\otimes e_n\right\|_{S_p^m[E_\lambda]}
       \sum_{n=1}^N\left\|y_n-z_n\right\|_p.
 \end{aligned}
\end{equation}
By the triangle inequality,
\eqref{eq:endpoint-disjoint-estimate} and \eqref{eq:synthesis-perturbation}, the map
\[
 T:E_\lambda\longrightarrow F,\qquad Te_n\coloneqq y_n,
\]
satisfies
\[
 \left\|T\right\|_{\mathrm{cb}}
 \le\sup_n\left\|z_n\right\|_p+\sum_{n=1}^\infty\left\|y_n-z_n\right\|_p<\infty.
\]
By the orthonormality of $(y_n)$ in $H$,
\[
 \Phi Te_k=(\langle y_j,y_k\rangle_H)_j=e_k,
 \qquad \Phi T=I_{E_\lambda}.
\]
Hence $T$ is a complete isomorphism onto the closed subspace
$Z\coloneqq T(E_\lambda)$, with inverse $\Phi|_Z$, and
\[
 P\coloneqq T\Phi:F\longrightarrow F,
 \qquad P^2=P,\qquad P(F)=Z,\qquad\left\|P\right\|_{\mathrm{cb}}<\infty.\qedhere
\]
\end{proof}

\subsection{Proof of Corollary~\ref{cor:finite-models}}
\label{sec:proof-finite-corollary}

\begin{proof}
We first reduce to the case where $\mathcal M$ has a faithful normal
tracial state $\tau$. Choose a separable infinite-dimensional closed
subspace $F_0\subset F$. By \eqref{eq:sigma-finite-support-reduction},
there is a $\sigma$-finite projection $e\in\mathcal M$ such that
$F_0\subset L_p(e\mathcal M e)$, and the finite algebra $e\mathcal M e$
admits a faithful normal tracial state
\cite[Corollary~III.2.5.8]{Blackadar}.
Since $x\mapsto exe$ is completely contractive, any completely bounded
projection $Q_0:L_p(e\mathcal M e)\to Z\subset F_0$ gives the required
projection $x\mapsto Q_0(exe)$ on $L_p(\mathcal M)$.
We may therefore assume that $\tau(\mathbf{1})=1$ and identify
$L_q(\mathcal M)$ with $L_q(\mathcal M,\tau)$.

By Theorem~\ref{thm:general-high-p-relative-complement} and
\cite[Corollary~6.10]{JX}, $F$ contains an infinite-dimensional closed
subspace $Z_0$ completely isomorphic to either $\ell_p$ or
$D\coloneqq C_p\cap R_p$.
If $Z_0$ is completely isomorphic to $\ell_p$, then
\cite[Corollary~6.3]{RX}, applied to $Z_0$, gives a subspace
$Z\subset Z_0$ completely isomorphic to $\ell_p$ and completely
complemented in $L_p(\mathcal M)$.
It remains to treat the case $Z\coloneqq Z_0\cong_{\mathrm{cb}}D$.

Fix a complete isomorphism $T:D\to Z$.
Since $Z$ is Hilbertian and contains no copy of $\ell_p$,
\cite[Theorem~5.1(iii), (v)]{RX}, with $q=2$ and
$h_0=\mathbf{1}$, gives $c>0$ such that
\begin{equation}\label{eq:finite-L2-equivalence}
 c\left\|z\right\|_p\le\left\|z\right\|_2\le\left\|z\right\|_p\qquad(z\in Z).
\end{equation}
Thus $Z$ is closed in $L_2(\mathcal M)$. Let
$P_2:L_2(\mathcal M)\to Z$ be the orthogonal projection, and regard
\[
 W\coloneqq T^{-1}P_2:L_2(\mathcal M)\longrightarrow D
\]
as a bounded operator between Hilbert spaces, with $D$ carrying its
canonical $\ell_2$ inner product. For $u_n\coloneqq W^*e_n$, we have
\[
 \sum_{n=1}^\infty|\langle u_n,\xi\rangle_2|^2
 =\left\|W\xi\right\|_D^2
 \le\left\|W\right\|^2\left\|\xi\right\|_2^2\qquad(\xi\in L_2(\mathcal M)).
\]
Lemma~\ref{lem:weighted-analysis}, applied with $h=\mathbf{1}$,
therefore shows that the map
\[
 Sx\coloneqq Wx=(\langle u_n,x\rangle_2)_n,
 \qquad x\in L_p(\mathcal M),
\]
is completely bounded into both $C_p$ and $R_p$.
By \eqref{eq:intersection-matrix-norm},
$S:L_p(\mathcal M)\to D$ is completely bounded as well.
Since $ST=I_D$, the map
\[
 P\coloneqq TS:L_p(\mathcal M)\longrightarrow Z
\]
is a completely bounded projection onto $Z$.\qedhere
\par
\end{proof}

\end{document}